\documentclass[11pt]{article}

\usepackage[T1]{fontenc}
\usepackage{amsfonts}
\usepackage{amsmath}
\usepackage{amssymb}
\usepackage{amsthm}
\usepackage{bbm}
\usepackage{bm}
\usepackage{mathrsfs}
\usepackage{verbatim}
\usepackage{setspace}
\usepackage{color}
\usepackage{pdfsync}
\usepackage{enumitem}
\usepackage{graphicx}
\usepackage{stmaryrd}
\usepackage{float} %place figure in text

\theoremstyle{plain}
\newtheorem{theorem}{Theorem}[section]
\newtheorem{proposition}[theorem]{Proposition}
\newtheorem{lemma}[theorem]{Lemma}
\newtheorem{corollary}[theorem]{Corollary}

\theoremstyle{definition}
\newtheorem{definition}[theorem]{Definition}
\newtheorem{remark}[theorem]{Remark}

\newtheorem{example}[theorem]{Example}

\theoremstyle{remark}
\renewenvironment{thebibliography}[1]{%
\begin{oldthebibliography}{#1}%
\setlength{\baselineskip}{.9em}
\linespread{1}
\small
\setlength{\parskip}{0.3ex}%
\setlength{\itemsep}{.3em}%
}%
{%
\end{oldthebibliography}%
}
\newcommand{\eps}{\varepsilon}

\newcommand{\R}{\mathbb{R}}

\newcommand{\cF}{\mathcal{F}}

\newcommand{\cI}{\mathcal{I}}

\newcommand{\cR}{\mathcal{R}}

\DeclareMathOperator*{\argmax}{arg\, max}

\newcommand{\as}{\mbox{-a.s.}}

\newcommand{\1}{\mathbf{1}}

\newcommand{\floor}[1]{\lfloor #1 \rfloor}
\newcommand{\ceil}[1]{\lceil #1 \rceil}

\numberwithin{equation}{section}

\usepackage[pdfborder={0 0 0}]{hyperref}
\hypersetup{
  urlcolor = black,
  pdfauthor = {Marcel Nutz, Yuchong Zhang},
  pdfkeywords = {Mean field game; rank-based reward; optimal contract; R&D competition},
  pdftitle = {A Mean Field Competition},
  pdfsubject = {A Mean Field Competition},
  pdfpagemode = UseNone
}
\begin{document}

\title{\vspace{-1.5em}
  A Mean Field Competition
\date{\today}
\author{
  Marcel Nutz%
  \thanks{
  Departments of Statistics and Mathematics, Columbia University, mnutz@columbia.edu. Research supported by an Alfred P.\ Sloan Fellowship and NSF Grant DMS-1512900.
  }
  \and
  Yuchong Zhang%
  \thanks{Department of Statistics, Columbia University, yz2915@columbia.edu. Research supported by NSF Grant DMS-1714607.
  }
 }
}
\maketitle \vspace{-1.2em}

\begin{abstract}
We introduce a mean field game with rank-based reward: competing agents optimize their effort to achieve a goal, are ranked according to their completion time, and paid a reward based on their relative rank. First, we propose a tractable Poissonian model in which we can describe the optimal effort for a given reward scheme. Second, we study the principal--agent problem of designing an optimal reward scheme. A surprising, explicit design is found to minimize the time until a given fraction of the population has reached the goal.
\end{abstract}

\vspace{0.9em}

{\small


}

%%%%%%%%%%%%%%%%%%%%%%%%%%%%%%%%%%%%%
\section{Introduction}\label{se:intro}

In this paper, we introduce two game-theoretic problems. The first one is a mean field game with infinitely many players that compete to obtain a reward. The second one is a principal--agent problem where the principal interacts with these agents; namely, the principal aims to distribute a given reward budget to the different ranks such as to minimize the time until the agents complete their task.

Let us think of the agents as independent research teams trying to develop a result or product in the same field. Following the literature on dynamic research and development (R\&D) detailed below, attaining a result will be modeled as a binary event. At any time $t$, each agent chooses a research effort $\lambda$ for which a quadratic instantaneous cost $c\lambda^{2}$ is to be paid, where $c>0$ is assumed to be constant for the purpose of this Introduction. In a Poissonian fashion, the agents' probability of reaching the goal in a small time interval $\Delta t$ is then given by $\lambda \Delta t + o(\Delta t)$; in the R\&D literature, $\lambda$ is sometimes interpreted as the accumulation of private knowledge and the goal is to file a patent. The agents are ranked according to their completion times and paid a reward $R(r)$ for rank $r$, where the reward scheme $R$ is a given decreasing\footnote{Decreasing and increasing are understood in the non-strict sense throughout the paper.} function. At any time, the agents observe the fraction $\rho(t)$ of players that have already completed the task and thus which portion of $R$ is still available; more precisely, the agents use feedback controls $\lambda(\rho(t))$. This rank-based coupling of the agents' optimization problems is a non-standard example of a mean field interaction.
We shall show that given $R$, this game has a unique Nash equilibrium when agents optimize the expectation of reward minus cost. In fact, this setting turns out to be very tractable: Theorem~\ref{th:equilibrium} provides explicit formulas for the equilibrium optimal control $\lambda^{*}$ and the agents' value function. These quantities are independent of the cost $c$ which, in the body of the paper, is also allowed to depend on the state $\rho(t)$ to model that the cost may diminish as more results become available.

The second problem is built on top of the first: as we have seen that any reward scheme $R$ leads to a unique equilibrium between the agents, we can study the problem of a manager or policy maker who would like to advance research. More precisely, we aim to minimize the time $T^{*}_{\alpha}$ until a given fraction $\alpha$ of the population has completed their task. The principal has a fixed reward budget $B=\int_{0}^{1} R(r)\,dr$ but may choose the shape of the decreasing function $R$; that is, how much reward to allocate to each rank. Quite surprisingly, the principal's optimization problem has an explicit but nontrivial solution (Theorem~\ref{th:principal}):
$$
  R^*(r)= \frac{B}{C'}\left\{\frac{1}{\sqrt{2-r}}+\frac{1}{2}\log\frac{(1+\sqrt{2-\alpha})(1-\sqrt{2-r})}{(1-\sqrt{2-\alpha})(1+\sqrt{2-r})}\right\}\1_{[0,\alpha]}(r),
$$
where $C'$ is a constant such that the budget constraint is saturated. As can be seen in Figure~\ref{fig:opt_rwd}, this function has two main features. The first one is a discontinuity at $r=\alpha$: a substantial amount is awarded to the last few relevant agents, but it is optimal to pay zero reward to the ranks after $\alpha$. While it is clearly important to incentivize the last agents that will complete the $\alpha$ fraction, these agents are not too discouraged by the fact that they may miss the rewarding ranks. The second feature is the shape of $R^*$ on $[0,\alpha]$. A priori, it may not even be obvious if it is better to provide a strictly decreasing reward, compared to paying the same amount to the first $\alpha$ ranks. It turns out that $R^*$ is decreasing, even if not very much so, and moreover it is concave. Thus, the difference in reward between two equidistant ranks increases later in the game, apparently to incentivize the remaining agents to choose a higher effort as shown on the second panel of Figure~\ref{fig:opt_rwd}.

\begin{figure}%[H]
\centering
\hspace*{-0.65in}
\includegraphics[height=4cm]{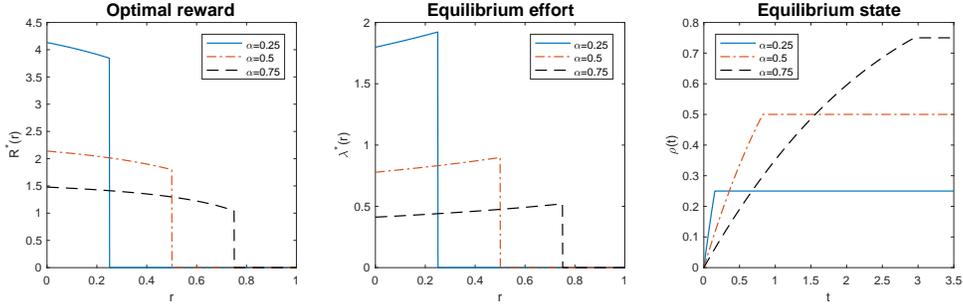}
\caption{The optimal reward scheme $R^{*}$ and the corresponding equilibrium effort $\lambda^{*}$ and state process $\rho$ for three different cut-off values $\alpha$, with $B=1$ and $c\equiv 1$.} %The minimal completion times are $T^*_{0.25}=0.1546$, $T^*_{0.5}=0.8256$ and $T^*_{0.75}=2.9477$.}
\label{fig:opt_rwd}
\vspace{-1em}
\end{figure}

While formulating our game with a continuum of agents is convenient to get directly to our main results, a more fundamental justification for the mean field game is to study an $N$-player game for $N\to\infty$. Indeed, we establish that the $N$-player versions of our two problems have unique solutions, albeit somewhat less explicit than in the mean field case. We show that the $N$-player equilibrium  converges to the mean field limit; that is, the value functions and the optimal feedback controls converge (Theorem~\ref{thm:fwd-convergence}) if the given reward schemes converge. Moreover, the optimal reward schemes for the principal and the corresponding expected completion times converge (Theorem~\ref{thm:principal-convergence}). The analysis for finite $N$ also allows us to study size effects which cannot be observed in the mean field limit and thus are rarely addressed for large population games. In particular, we shall observe in Section~\ref{se:sizeEffects} that an increase in population size adversely affects the principal: the minimal expected completion time for a fixed target proportion $\alpha$ is increasing in $N$.

Many of our results make crucial use of the explicit or semi-explicit formulas that can be obtained in our model; in fact, much of our work has been devoted to finding a setting that is tractable. On the other hand, one may suspect that qualitative features of the results, in particular the shape of the optimal reward scheme $R^{*}$, would be similar even if the precise mechanism of the competition were different. At this time, it seems that we do not have the tools to address this question which therefore remains an interesting direction for future research.

%%%%%%%%%%%%%%%%%%%%%%
\subsection{Literature}

Dynamic competitions (also called races) are classical in the Economics literature. An early reference related to our paper is Reinganum~\cite{Reinganum.82} which discusses an $N$-player dynamic game of R\&D and patent protection. Rewards are paid at a fixed time horizon and two cases are studied: either only the first-ranked player is rewarded (perfect patent protection), or the subsequent ranks receive a positive, smaller reward; however, only the case of an identical reward for all ``imitators'' is considered. Malueg and Tsutsui~\cite{MaluegTsutsui.97} extend Reinganum's setting with a hazard rate that represents changes in the difficulty of the research project, an aspect we have incorporated differently by using a state-dependent cost that can model how the increase in public knowledge affects the project.
Harris and Vickers~\cite{HarrisVickers.87} and Grossman and Shapiro~\cite{GrossmanShapiro.87} focus on the strategic interactions between agents in multiperiod 2-player games. A recent work in this area is Cao~\cite{Cao.14} which studies a continuous-time, continuous-state version of a model in~\cite{HarrisVickers.87}. We refer to \cite{Cao.14,MaluegTsutsui.97} for further references in this literature.

Mean field games were introduced by Lasry and Lions~\cite{LasryLions.06a,LasryLions.06b,LasryLions.07} and Huang, Malham\'e, and Caines~\cite{HuangMalhameCaines.07, HuangMalhameCaines.06} to study Nash equilibria in the limiting regime where the number of players tends to infinity and interactions take place through the empirical distribution of the private states; we refer to Gu\'eant, Lasry and Lions \cite{GueantLasryLions.11}, Bensoussan, Frehse and Yam~\cite{BensoussanFrehseYam.13} and Carmona and Delarue~\cite{CarmonaDelaRue.17a,CarmonaDelaRue.17b} for background on mean field games. Since the impact of an individual player on the aggregate distribution is negligible, finding a Nash equilibrium reduces to solving a stochastic optimization problem for a representative player against a fixed environment, together with a consistency condition. This can be justified rigorously by formulating a game with a continuum of players and by showing that it is the limit of an $N$-player game. Convergence is often shown backwards; i.e., the mean field equilibrium is shown to provide an $\eps$-Nash equilibrium for the $N$-player game. The forward convergence of the $N$-player equilibrium to the mean field equilibrium is typically more difficult to prove. For standard, diffusion-driven mean field games, this has been accomplished recently in the seminal work of Cardaliaguet, Delarue, Lasry and Lions~\cite{CardaliaguetDelarueLasryLions.15}. While our game is of a different form, its tractability allows us to give an elementary yet nontrivial proof of the forward convergence; one feature in common with~\cite{CardaliaguetDelarueLasryLions.15} is that we use feedback controls. In a finite-state setting, forward convergence was shown by Gomes, Mohr and Souza \cite{GomesMohrSouza.13} for a small time horizon and recently by Bayraktar and Cohen~\cite{BayraktarCohen.17} for an arbitrary finite time horizon.

Competitions, i.e., rank-based rewards, are a classical topic in contract theory, dating back to the work of Lazear and Rosen~\cite{LazearRosen.81}. 
With applications from school grades to sports and business competitions, rank-order prize allocation is one of the most widely used relative performance evaluation criteria; we refer to Vojnovi\'c~\cite{Vojnovic.15} for a detailed introduction and an extensive list of references. To the best of our knowledge, the only existing work on mean field games with rank-based reward is Bayraktar and Zhang~\cite{BayraktarZhang.16} where players are ranked according to their terminal positions. The main aim of that work is to obtain abstract existence results for games with common noise via translation invariance, whereas in the present work players are ranked according to exit times and the focus is on specific properties of solutions (and, of course, the principal's problem). In a different way, exit times are used in the toy example ``When does the meeting start?'' of~\cite{GueantLasryLions.11}. Rank-based features are also studied in the literature on (uncontrolled) particle systems; one example is Shkolnikov~\cite{Shkolnikov.12}. Nadtochiy and Shkolnikov~\cite{NadtochiyShkolnikov.17} consider particles interacting through hitting times. A different but related recent literature studies mean field games of timing where players directly choose stopping times; see  Carmona and Lacker~\cite{CarmonaDelaRueLacker.17}, Bertucci~\cite{Bertucci.17} and Nutz~\cite{Nutz.16}.

Continuous-time principal--agent problems with multiple agents have been studied by Koo, Shim and Sung~\cite{KooShimSung.08} and {Elie and Possama\"{i}~\cite{EliePossamai.16}, and extended to the mean field setting by Elie, Mastrolia and Possama\"{i}~\cite{ElieMastroliaPossamai.16} and Bensoussan, Chau and Yam~\cite{BensoussanChauYam.16}. While these works have not considered rank-based rewards, a common feature is the Stackelberg equilibrium: the principal designs a reward scheme which the agents take as an external input to form a Nash equilibrium among themselves. By contrast, in mean field games with a major player as in Huang~\cite{Huang.09} or Carmona and Wang~\cite{CarmonaWang.16}, a Nash equilibrium is formed collectively by the major and minor players. To the best of our knowledge, the convergence of the $N$-player principal--agent problem to the mean field limit has only been established in a simple example where the equilibrium controls are independent of $N$; see~\cite{ElieMastroliaPossamai.16}.

The remainder of this paper is structured as follows. In Section~\ref{se:equilib}, we determine the unique Nash equilibrium of the mean field competition for a continuum of players with a given reward scheme. On the strength of this result, Section~\ref{se:MFprincipal} solves the associated principal--agent problem where the principal designs the reward scheme. In Section~\ref{se:NplayerProblems} we study the corresponding $N$-player problems and Section~\ref{se:convergence} establishes their convergence as $N\to\infty$. Proofs are gathered in Appendix~\ref{se:proofs}, whereas Appendix~\ref{se:exactLLN} provides  background on the Exact Law of Large Numbers used in Section~\ref{se:equilib}.

%%%%%%%%%%%%%%%%%%%%%%%%%%%%%%%%%%%%%
\section{Mean Field Game}\label{se:equilib}

Let $(I,\cI,\mu)$ be an atomless probability space; each $i\in I$ is thought of as an agent. Moreover, let $(\Omega,\cF,P)$ be another probability space, to be used as the sample space. Let $(Z^{i})_{i\in I}$ be a family of exponential(1)-distributed random variables on $\Omega$ which is essentially pairwise independent; that is, for $\mu$-almost all $i\in I$, $Z^{i}$ is independent of $Z^{j}$ for $\mu$-almost all $j\in I$. We assume that this family is defined on an extension of the product $(I\times\Omega,\cI\otimes \cF,\mu\otimes P)$ for which the Exact Law of Large Numbers holds, as detailed in Appendix~\ref{se:ELLN}.
Given a locally Lebesgue-integrable function $\theta:\R\to [0,\infty)$, we define 
\begin{equation}\label{eq:arrivalTimeDefn}
  \tau^{i}_{\theta}=\inf \left\{t:\, \int_{0}^{t} \theta(s)\,ds = Z^{i}\right\};
\end{equation}
then $(\tau^{i}_{\theta})_{i\in I}$ are essentially pairwise independent and their distribution corresponds to the first jump time of an inhomogeneous Poisson process with intensity $\theta$. Below, the function $\theta$ is of the form $\theta=\lambda\circ\rho$ where $\lambda$ is a function chosen by the agent and $\rho$ is a given function, and we shall find it convenient to write $\tau^{i}_{\lambda}$ for $\tau^{i}_{\theta}$ despite the abuse of notation. If $\tau^{i}_{\lambda}\leq t$, we shall say that agent $i$ has ``arrived'' by time $t$.

We define an \emph{admissible (feedback) control} as a piecewise Lipschitz continuous\footnote{That is, $[0,1)$ is the union of finitely many intervals on which $\lambda$ is Lipschitz.} function $\lambda: [0,1)\to\R_{+}$. The next lemma introduces the state process that emerges if all agents use the control $\lambda$.

\begin{lemma}\label{le:kolmogorovEqn}
  Let $\lambda\in\Lambda$ be an admissible feedback control. There exists a unique continuous function $\rho: \R_{+}\to [0,1)$ satisfying
  \begin{equation}\label{eq:kolmogorov}
    \rho(t)=\int_{0}^{t}\lambda(\rho(s))(1-\rho(s))\,ds,\quad t\geq 0.
  \end{equation}
  If all agents use the feedback control $\lambda$, then
  $\rho(t) = \mu\{i:\, \tau_{\lambda}^{i}(\omega)\in [0,t]\}$ $P\as$ as well as 
  $\rho(t) = P\{\tau_{\lambda}^{i}\in [0,t]\}$ $\mu\as$;
%  \end{align*}
%  \begin{align*}
%    \rho(t) &= \mu\{i:\, \tau_{\lambda}^{i}(\omega)\in [0,t]\}\quad P\as,\\
%    \rho(t) &= P\{\tau_{\lambda}^{i}\in [0,t]\}\quad \mu\as;
%  \end{align*}
  that is, $\rho(t)$ is both the proportion of agents that have arrived by time~$t$ and the probability that any given agent~$i$ has arrived by time~$t$.
\end{lemma}

Next, we fix a \emph{cost coefficient} $c:[0,1]\to (0,\infty)$ which is assumed to be Lipschitz continuous (thus, $c$ and $1/c$ are bounded). Moreover, we fix a \emph{reward scheme} $R: [0,1]\to\R_{+}$ which is assumed to be decreasing, piecewise Lipschitz continuous, and left-continuous at $r=1$. We interpret $R(r)$ as the reward paid to an agent arriving at rank $r$.

Let us now consider the control problem of a given agent~$i$ before arriving, assuming that a proportion~$r$ of the population has already arrived and that all other agents transition according to a deterministic function $\rho$. Then, our agents' value function is 
\begin{equation}\label{eq:valueSinglePlayer}
    v(r)=\sup_{\lambda\in\Lambda} E\left[R(\rho(\tau_{\lambda}))-\int_0^{\tau_{\lambda}} c(\rho(t)) \lambda(\rho(t))^2\,dt\bigg|\rho(0)=r\right],
\end{equation}
where $\tau_{\lambda}=\tau_{\lambda}^{i}$ is the arrival time of the given agent for the control $\lambda$. Here, we use the convention that $\rho(\infty):=1$, meaning that agents who never arrive are paid the reward $R(1)$. If $\lambda\in\Lambda$ attains the supremum in~\eqref{eq:valueSinglePlayer}, we say that $\lambda$ is an \emph{optimal control given $\rho$}.

If $\lambda$ is an optimal control given the induced function $\rho$ defined by~\eqref{eq:kolmogorov}, we say that $\lambda$ is an \emph{equilibrium optimal control} and $\rho$ is the corresponding \emph{equilibrium state process}. This is a Nash equilibrium: if all other players use the feedback control $\lambda$, the state evolves according to $\rho$ by Lemma~\ref{le:kolmogorovEqn} (recall that $\mu$ is atomless), and then $\lambda$ is an optimal control for our fixed player.

\begin{theorem}\label{th:equilibrium}
  Let $R$ be a reward scheme. Then there exists a unique (a.e.) equilibrium optimal control $\lambda^*\in\Lambda$, given by
  \begin{equation}\label{eq:equilibOptControl}
    \lambda^*(r)=\frac{R(r)-\frac{1}{2\sqrt{1-r}}\int_r^1 \frac{R(y)}{\sqrt{1-y}}\,dy}{2c(r)},\quad r\in[0,1) 
  \end{equation}
  and the corresponding equilibrium state process $\rho$ is determined by~\eqref{eq:kolmogorov} with $\lambda=\lambda^{*}$. In equilibrium, the value function of any agent before arriving is 
  \begin{equation}\label{eq:value}
    v(r)=\frac{1}{2\sqrt{1-r}}\int_{r}^1 \frac{R(y)}{\sqrt{1-y}}\,dy,\quad r\in[0,1).
  \end{equation}
\end{theorem}

Let us remark that while the piecewise Lipschitz requirement is mainly for convenience, the continuity of $R$ at $r=1$ is more essential in providing existence. The following example exhibits a phenomenon that is familiar in infinite-horizon optimal stopping problems.

\begin{example}\label{ex:counterExExistence}
  Suppose that $c\equiv1$ and $R=\1_{[0,1)}$; that is, the reward is one for agents arriving in finite time and zero for those who never arrive. Then, using the constant control $\lambda\equiv\eps>0$ yields an exponential arrival time with $E[\tau]=1/\eps$ and thus the expected reward is
  $$
    E\left[R(\rho(\tau))-\int_0^{\tau} \eps^2\,dt\right] = 1-\eps.
  $$
  As a result, the value function satisfies $v(r)=1=R(r)$ for all $r<1$. But since $\lambda\equiv0$ yields zero reward and any other control has positive cost, this value is not attained:  there is no optimal control, and thus no equilibrium in the above sense.
\end{example}

\begin{remark}\label{rk:valueIndeCost}
  We see from~\eqref{eq:value} that the equilibrium value function is independent of the cost coefficient $c$. This can also be understood directly by expressing~\eqref{eq:valueSinglePlayer} as an integral over the ranks, using the Kolmogorov equation~\eqref{eq:kolmogorov} and the change-of-variable formula:
  $$%\begin{equation}\label{eq:equilibValueRankspace}
    v(r_{0})= \sup_{\lambda\in\Lambda} E\left[R(\rho(\tau_{\lambda}))-\int_{\rho(0)}^{\rho(\tau_{\lambda})} c(r) \lambda(r)\,\frac{dr}{1-r} \bigg|\rho(0)=r_{0}\right].
  $$%\end{equation}
  Indeed, $\rho(\tau_{\lambda})$ is independent of $\lambda$---when all agents use the same control, their ranking is given by the ranking of the $Z^{i}$. On the other hand, $c\lambda\in\Lambda$ if and only if $\lambda\in\Lambda$, and hence $v$ is independent of $c$ in equilibrium. Intuitively, a higher cost leads to a smaller optimal effort, but since this holds for all agents, the equilibrium state $\rho$ is slowed down to the extent that the reduced effort results the same reward.
\end{remark}

The equilibrium value function of any agent has a surprising interpretation: it can be compared to a deal where the agent pays no cost for his (constant) effort but is given the handicap of running at half the intensity of the competitors.

\begin{proposition}\label{pr:valueProbabRepresent}
  The equilibrium value function $v$ of~\eqref{eq:value} coincides with the value function of an agent whose effort is fixed at $\lambda\equiv\lambda_{0}\in(0,\infty)$ and is charged zero cost, while all other agents use $\lambda\equiv 2\lambda_{0}$:
  $$%\begin{equation}\label{eq:valueSinglePlayerRepresent}
    v(r)= E[R(\rho(\tau))|\rho(0)=r],\quad\mbox{where $\tau\sim Exp(\lambda_{0})$ and $\rho'(t)=2\lambda_{0}(1-\rho(t))$.}
  $$
In particular, $v(0)=E[R(1-e^{-2\tau})]$ for $\tau\sim Exp(1)$.
\end{proposition}

The following result shows that the unique equilibrium of Theorem~\ref{th:equilibrium} is stable with respect to the reward scheme.

\begin{proposition}\label{pr:stability}
  Let $R_{n},R$ be reward schemes such that $R_{n}\to R$ pointwise. Then the corresponding equilibrium optimal controls also converge pointwise, whereas the equilibrium value functions and state processes converge uniformly.
\end{proposition}

%%%%%%%%%%%%%%%%%%%%%%%%%%%%%%%%
\subsection{Examples with Closed-Form Solutions}

In this section, we present a family of explicitly solvable examples. Given a total reward budget $B=\int_0^1 R(r)\,dr\geq0$, the family has a cut-off parameter $\alpha\in(0,1]$ indicating that no reward will be paid to agents ranked lower than~$\alpha$, as well as a shape parameter $q\ge 0$. The general form is then given by
\[
  R(r)=\kappa(1-r)^q \1_{[0,\alpha]}(r), \qquad \kappa=\frac{B(1+q)}{1-(1-\alpha)^{1+q}};
\] 
the constant $\kappa$ is chosen such that $B=\int_0^1 R(r)\,dr$.
We notice that a larger value of $q$ indicates that a larger portion of the reward budget is paid to highly ranked player, whereas $q=0$ corresponds to a uniform distribution of the reward among the top $\alpha$ ranks.
For such a reward, the value function and the optimal effort of Theorem~\ref{th:equilibrium} admit closed-form solutions:
\[
  %v(r)=\frac{B(1+q)}{(1+2q)[1-(1-\alpha)^{1+q}]}\left((1-r)^q-(1-\alpha)^q\sqrt{\frac{1-\alpha}{1-r}}\right)^+,
  v(r)=\frac{\kappa}{(1+2q)}\left((1-r)^q-(1-\alpha)^q\sqrt{\frac{1-\alpha}{1-r}}\right)^+,
\]
\[
  \lambda^*(r)=\1_{\{r\le \alpha\}}\frac{\kappa}{2c(r)(1+2q)}\left(2q(1-r)^q+(1-\alpha)^q\sqrt{\frac{1-\alpha}{1-r}}\right).
\]
In the boundary case $\alpha=1$, the Lipschitz assumption of Theorem~\ref{th:equilibrium} is not satisfied when $0<q<1$. However, one can check the indicated formulas by direct computation in this case.

In general, the c.d.f.\ $F_{\tau_{\lambda^*}}(t)=\rho(t)$ of any agent's equilibrium completion time can be computed numerically by solving the Kolmogorov equation~\eqref{eq:kolmogorov} for~$\rho$. Inverting the equilibrium state process also gives rise to the quantile $T_\beta = \inf\{t:\, \rho(t)\geq \beta\}$; that is, the time until a $\beta$-proportion of the players has reached the goal. In the following special cases, these quantities can be obtained in closed form.

%%%%%%%%%%%%%%%%%%%%%%%%%%%%%%%%%%%%%%%%
\subsubsection{Power Reward Without Cut-off}

This case corresponds to $\alpha=1$, where we also assume that the cost $c$ is constant. Then the above formulas specialize to 
\[
  v(r)=\frac{B(1+q)}{1+2q}(1-r)^q, \qquad \lambda^*(r)=\frac{Bq(1+q)}{c(1+2q)}(1-r)^q,
\]
and we can also solve for 
$
  F_{\tau_{\lambda^*}}(t)=\rho(t)=1-\big(1+\frac{Bq^2(1+q)}{c(1+2q)}t\big)^{-\frac{1}{q}}
$
and its $\beta$-quantile
$
  T_\beta=\frac{c(1+2q)}{Bq^2(1+q)}\left[(1-\beta)^{-q}-1\right].
$
We see that the equilibrium value $v$ is decreasing in $q$. That is, each individual is worse off if the reward scheme heavily favors the highly-ranked players; this can be attributed to the cost caused by the large effort level $\lambda^{*}$ in the beginning of the competition. We also observe that $\lambda^{*}$ is decreasing in $r$ so that agents decrease their effort once the higher ranks are filled.

%\begin{figure}[H]
%\centering
%\includegraphics[height=14cm]{codes/opt_effort1.eps}
%\caption{Optimal effort under power reward without cut-off, assuming unit budget and cost coefficient.}
%\label{fig:opt_effort1}
%\end{figure}

%%%%%%%%%%%%%%%%%%%%%%%%%%%%%%%%%%%%%%%%
\subsubsection{Uniform Reward with Cut-off}

This case corresponds to $q=0$, and we again assume that the cost $c$ is constant. The general formulas now specialize to
\[
  v(r)=\frac{B}{\alpha}\left(1-\sqrt{\frac{1-\alpha}{1-r}}\right)^+, 
  \qquad
  \lambda^*(r)=\1_{\{r\le \alpha\}}\frac{B}{2c \alpha}\sqrt{\frac{1-\alpha}{1-r}}.
\]
We also have
$F_{\tau_{\lambda^*}}(t)= \rho(t)= 1-\left(1-\frac{B\sqrt{1-\alpha}}{4c\alpha}t\right)^2$ for $t\le T_{\alpha}$ and 
$F_{\tau_{\lambda^*}}(t)= \rho(t)=\alpha$ for $t> T_{\alpha}$,
%\[
%  F_{\tau^*}(t)= \rho^*(t)= 
%  \begin{cases}
%    1-\left(1-\frac{B\sqrt{1-\alpha}}{4c\alpha}t\right)^2, & t\le T_{\alpha}, \\
%    \alpha, & t> T_{\alpha},
%  \end{cases}
%\]
where
\begin{equation}\label{eq:quantileUniformCutoff}
  T_{\alpha}=\frac{4c\alpha(1-\sqrt{1-\alpha})}{B\sqrt{1-\alpha}},
\end{equation}
and then the general quantile is
$T_{\beta}=\frac{4c\alpha (1-\sqrt{1-\beta})}{B\sqrt{1-\alpha}}$ for $\beta\le \alpha$ and $T_{\beta}=\infty$ for $\beta> \alpha$.
%$$
%  T_{\beta}=
%  \begin{cases}
%    \frac{4c\beta(1-\sqrt{1-\beta})}{B\sqrt{1-\beta}},& \beta\le \alpha\\
%    \infty, &  \beta> \alpha.
%  \end{cases}
%$$
In contrast to the case $\alpha=1$, we see that $\lambda^*$ is increasing in $r$ for $r\leq\alpha$: as the race progresses, the agents compete for the remaining reward and increase their effort up to the time when an $\alpha$-proportion of agents has reached the goal, and then the remaining players give up; cf.\ Figure~\ref{fig:opt_effort2}.

\begin{figure}[H]
\centering
\vspace{1em}
\includegraphics[height=4.5cm]{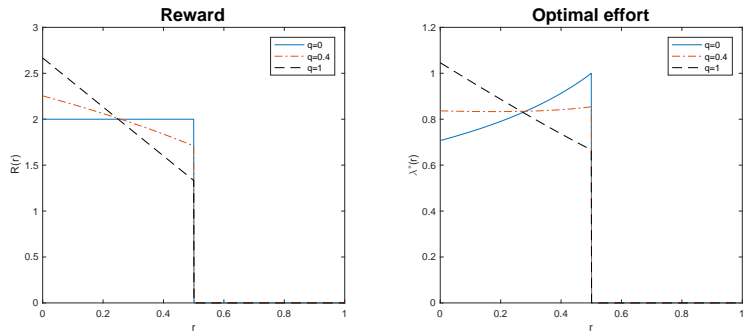}
\caption{Optimal effort under power reward with cut-off $\alpha=0.5$, assuming $B=1$ and $c\equiv 1$.}
\label{fig:opt_effort2}
\end{figure}

\subsection{Staircase Reward}

Consider a reward scheme $R$ and a cost coefficient $c$ of the staircase form
\begin{align*}
  R=R_1 \1_{[r_0,r_1]}+\sum_{j=2}^n R_j \1_{(r_{j-1},r_j]}, \qquad
  c=c_1 \1_{[r_0,r_1]}+\sum_{i=2}^n c_j \1_{(r_{j-1},r_j]},
\end{align*}
where $R_1\ge R_2 \ge \cdots \ge R_n\geq0$ and $0=r_0<r_1<\cdots<r_n=1$ are constants. The formulas~\eqref{eq:equilibOptControl} and~\eqref{eq:value} then yield that for $r_{j-1}<r\le r_j$,
\begin{align*}
  v(r)&=R_j+\frac{1}{\sqrt{1-r}}\left[-R_j\sqrt{1-r_j}+\sum_{k=j+1}^n R_k \left(\sqrt{1-r_{k-1}}-\sqrt{1-r_k}\right)\right],\\
  \lambda^*(r)&=\frac{1}{2c_j\sqrt{1-r}}\left[R_j\sqrt{1-r_j}-\sum_{k=j+1}^n R_k \left(\sqrt{1-r_{k-1}}-\sqrt{1-r_i}\right)\right].
\end{align*}
We claim that the equilibrium state $\rho$ is given by
\begin{equation}\label{eq:staircaseState}
  \rho(t)=1-\left(\sqrt{1-r_{j-1}}-\frac{A_j}{4c_j}(t-t_{j-1})\right)^2,\quad t_{j-1}\leq t \leq t_{j},
\end{equation}
where
$A_j=R_j\sqrt{1-r_j}-\sum_{k=j+1}^n R_k (\sqrt{1-r_{k-1}}-\sqrt{1-r_k})$ and
$t_{j}$ is recursively defined by $t_j=t_{j-1}+\frac{4c_j}{A_j}(\sqrt{1-r_{j-1}}-\sqrt{1-r_j})$ and $t_0=0$.
%\begin{align*}
%  t_j&=t_{j-1}+\frac{4c_j}{A_j}\left(\sqrt{1-r_{j-1}}-\sqrt{1-r_j}\right),\\
%  A_j&=R_j\sqrt{1-r_j}-\sum_{k=j+1}^n R_k \left(\sqrt{1-r_{k-1}}-\sqrt{1-r_k}\right).
%\end{align*}
This is to be read with the convention that $1/0=\infty$; indeed, we have $A_{j}=0$ and $t_{j}=\infty$ if (and only if) $R_{j}=R_{j+1}=\cdots=R_{n}$. 
To see \eqref{eq:staircaseState}, we may solve the ODE~\eqref{eq:kolmogorov} successively for each interval $[r_{j-1},r_{j}]$. Let $t_{0}=0$. Suppose we have already found $t_0,\ldots, t_{j-1}$ and $\rho(t)$ for $t\in [0,t_{j-1}]$. Then the ODE on the $j$th interval reads
$\rho'(t)=\frac{A_j}{2c_j}\sqrt{1-\rho(t)}$ with initial condition $\rho(t_{j-1})=r_{j-1}$
%\[
%  \rho'(t)=\frac{A_j}{2c_j}\sqrt{1-\rho(t)}, \quad \rho(t_{j-1})=r_{j-1}
%\]
and the solution is given by~\eqref{eq:staircaseState}, whereas $t_{j}$ is determined through the condition $\rho(t_{j})=r_{j}$.

Finally, let $\beta\in(0,1]$ be given. By adding $\beta$ to the grid if necessary, we may assume without loss of generality that $\beta=r_{j_0}$ for some $j_0\in\{1,\ldots,n\}$, and then the $\beta$-quantile is
$
  T_{\beta}=t_{j_0}=\sum_{j=1}^{j_0}\frac{4c_j}{A_j}(\sqrt{1-r_{j-1}}-\sqrt{1-r_j}).
$

%%%%%%%%%%%%%%%%%%%%%%%%%%%%%%%%%%%
\section{Mean Field Principal--Agent Problem}\label{se:MFprincipal}

We have seen that for a given reward scheme $R$, there exists a unique (deterministic) equilibrium  state $\rho$ and thus for $\alpha\in(0,1]$, the time 
$$
  T_{\alpha}(R) = \inf\{t\geq0:\, \rho(t)\geq \alpha\}\in (0,\infty]
$$
is deterministic and well-defined. This is the time until an $\alpha$-proportion of the population has reached the goal, or equivalently, $T_{\alpha}$ is the $\alpha$-quantile of the distribution of the equilibrium arrival time $\tau^{*}$.

In this section, we fix $\alpha\in(0,1)$ and the total reward budget $B>0$, and ask for a 
reward scheme $R$ which minimizes $T_{\alpha}(R)$ subject to the constraint that $\int_{0}^{1} R(r)\,dr\leq B$.
%\begin{equation}\label{eq:optimalContractProblem}
%\mbox{minimize }T_{\alpha}(R) \qquad\mbox{subject to}\qquad \int_{0}^{1} R(r)\,dr\leq B.
%\end{equation}
This corresponds to a principal--agent problem in the second-best sense: the planner can set the reward for the agents, but cannot dictate their choice of controls. The principal considers her project completed when an $\alpha$-proportion of the agents have reached their goal, and aims to find the minimal completion time
\begin{equation}\label{eq:optimalContractValue}
  T^{*}_{\alpha} = \inf_{R\in\cR:\, \int R(r)\,dr\leq B} T_{\alpha}(R),
\end{equation}
where $\cR$ is the set of all reward schemes.
We remark that for $\alpha=1$, we have $T_{\alpha}(R)=\infty$ for all $R$, whence we do not consider this case. On the other hand, $T_{\alpha}^{*}<\infty$ for all $\alpha\in(0,1)$ as this is already accomplished by the uniform reward $R$ with cut-off at $\alpha$; cf.~\eqref{eq:quantileUniformCutoff}. 

An additional assumption on the cost coefficient $c$ is needed for our result:
\begin{equation}\label{eq:assumptCost}
  r\mapsto \frac{c(r)(1-r)}{2-r}\quad \mbox{is decreasing}.
\end{equation}
This assumption is discussed in more detail in Remark~\ref{rk:cAssumt} below.
The solution to the principal's problem is then given as follows.

\begin{theorem}\label{th:principal}
  Let $c$ satisfy~\eqref{eq:assumptCost}. Given a  reward budget $B>0$ and $\alpha\in(0,1)$, there is an a.e.\ unique optimal reward scheme $R^{*}$ attaining the minimal completion time $T^*_{\alpha}$ of~\eqref{eq:optimalContractValue}, given by
  \begin{equation}\label{eq:optimalR}
    R^*(r)= \frac{B}{C}\left\{\sqrt{\frac{c(r)}{2-r}}+\frac{1}{2}\int_r^\alpha \frac{1}{1-s}\sqrt{\frac{c(s)}{2-s}}ds\right\} \1_{[0,\alpha]}(r),
  \end{equation}
  and the minimal completion time is
  \begin{equation}\label{eq:optimalTime}
    T^*_{\alpha}=\frac{4 C^2}{B},\quad\mbox{where}\quad 
  C=\frac{1}{2}\int_0^\alpha \frac{\sqrt{c(r)(2-r)}}{1-r}\,dr.
  \end{equation}
  The corresponding equilibrium effort is
  \[
    \lambda^*(r)=\frac{B}{2C}\frac{1}{\sqrt{(2-r)c(r)}}\1_{[0,\alpha]}(r).
  \]  
  In the particular case where the cost $c$ is constant, we have 
  \begin{align*}
    R^*(r)&= \frac{B}{C'}\left\{\frac{1}{\sqrt{2-r}}+\frac{1}{2}\log \frac{(1+\sqrt{2-\alpha})(1-\sqrt{2-r})}{(1-\sqrt{2-\alpha})(1+\sqrt{2-r})}\right\}\1_{[0,\alpha]}(r),\\
    T^*_{\alpha}&=\frac{4 cC'^2}{B}, \\
    C'&=\frac{C}{\sqrt{c}}=\sqrt{2}-\sqrt{2-\alpha}+\frac{1}{2}\log\frac{(1+\sqrt{2-\alpha})(1-\sqrt{2})}{(1-\sqrt{2-\alpha})(1+\sqrt{2})}.
  \end{align*}
\end{theorem}

Figure~\ref{fig:opt_rwd} (rendered in the Introduction) shows $R^{*}$, $\lambda^{*}$ and $\rho$ for constant cost coefficient $c$. As discussed, the key features of $R^{*}$ are the strict decrease and concavity on $[0,\alpha]$ and the discontinuity at $\alpha$. The equilibrium effort $\lambda^{*}$ is strictly increasing on $[0,\alpha]$.  For general $c$, the product $\sqrt{c}\lambda^{*}$ is increasing, but $\lambda^{*}$ need not be.

\begin{remark}\label{rk:cAssumt}
 Assumption~\eqref{eq:assumptCost} is satisfied in particular if $c$ is decreasing, which certainly holds in the applications we have in mind. If we suppose for simplicity that $c$ is differentiable, the assumption is equivalent to the derivative $c'$ satisfying
  $
    c'(r)\leq \frac{c(r)}{(2-r)(1-r)},
  $
  for which a sufficient condition is that $c'(r)\leq c(r)/2$. Thus, an increase of $c$ is permissible if it is  limited relative to the level of $c$. When the assumption is not satisfied, \eqref{eq:optimalR} no longer describes the solution to the principal--agent problem; in fact, \eqref{eq:optimalR} is not a decreasing function and hence not a reward scheme. The proof of Theorem~\ref{th:principal} shows that finding an optimal reward scheme can still be phrased as a convex optimization problem; however, the monotonicity constraint on $R$ is now binding which is an obstruction to finding an explicit solution.
\end{remark}

We may also ask the reverse question: given $\alpha\in(0,1)$ and a desired completion time $T>0$, what is the minimal budget enabling the principal to achieve $T$? The answer follows from Theorem~\ref{th:principal} by inverting~\eqref{eq:optimalTime}.

\begin{corollary}\label{co:optimalB}
Let $c$ satisfy~\eqref{eq:assumptCost}. Given $\alpha\in(0,1)$ and $T>0$, the minimal budget enabling the principal to achieve a completion time $T^{*}_{\alpha}\leq T$ is
  $
    B^*=4C^2/T,
  $
  where $C$ is given by~\eqref{eq:optimalTime}.
\end{corollary}

%%%%%%%%%%%%%%%%%%%%%%%%%%%%%%%%%%%
\section{$N$-Player Problems}\label{se:NplayerProblems}

In this section, we study a version of the competition with finitely many players as well as a corresponding version of the principal--agent problem. The connections between these and the mean field formulations will be established in Section~\ref{se:convergence}.

%%%%%%%%%%%%%%%%%%%%%%%%%%%%%%%%%%%
\subsection{The $N$-Player Game}

We consider a game with $N$ players, where $N\geq1$ is a fixed integer. At any time $t$, each player $i$ observes the number $n$ of players that have already arrived, and chooses an effort level $\lambda_{i}(n)\in\R_{+}$. Suppose that player $i$ uses the feedback control $\lambda_{i}$ and all other players use a feedback control $\lambda_{-i}$. We denote by $\xi_{\lambda_i, \lambda_{-i}}(t)$ the number of players that have arrived  by time~$t$; i.e.,
\[
  \xi_{\lambda_i,\lambda_{-i}}(t)=\sum_{i=1}^N \1_{\{\tau^i\le t\}},
\]
where 
$\tau^i=\inf\{t\ge 0: \int_0^t \lambda_i(\xi_{\lambda_i,\lambda_{-i}}(s))\,ds=Z^i\}$ is the arrival time of player~$i$ and  
$\tau^j=\inf\{t\ge 0: \int_0^t \lambda_{-i}(\xi_{\lambda_i,\lambda_{-i}}(s))\,ds=Z^j\}$ is the arrival time of player $j\ne i$, for some 
%\end{align*}
%\begin{align*}
%  \tau^i&=\inf\left\{t\ge 0: \int_0^t \lambda_i(\xi_{\lambda_i,\lambda_{-i}}(s))\,ds=Z^i\right\},\\
%  \tau^j&=\inf\left\{t\ge 0: \int_0^t \lambda_{-i}(\xi_{\lambda_i,\lambda_{-i}}(s))\,ds=Z^j\right\}, \quad j\ne i,
%\end{align*}
independent exponential random variables $\{Z^1, \ldots, Z^N\}$ with unit rate. The existence of the state process is clear in this case; we may see $(\1_{\{\tau^i\le t\}},\xi_{\lambda_i,\lambda_{-i}})$ as a Markov pure jump process with values in $\{0,1\}\times\{0,1,\dots,N\}$. We emphasize that we now use the number rather than the fraction of arrived players as the state variable.

Let $R_n\in\R_{+}$ be the reward for finishing at the $n$-th place; as before, we assume that $(R_{n})_{1\leq n\leq N}$ is decreasing, and we convene that $R_{N}$ is paid to players that never arrive. Moreover, let $c_n>0$ be the cost coefficient when $n$ players have arrived. Then the objective of player $i$ is to maximize
\[
  J_i(\lambda_i;\lambda_{-i})=E\left[R_{\xi_{\lambda_i,\lambda_{-i}}(\tau^i)}-\int_0^{\tau^i} c_{\xi_{\lambda_i,\lambda_{-i}}(s)}\lambda_i^2(\xi_{\lambda_i,\lambda_{-i}}(s))ds\right]
\]
and $\lambda$ is a (symmetric) equilibrium optimal control if $\argmax_{\lambda_i} J_i(\lambda_i;\lambda)=\lambda$ for all $i$. For $0\leq n \leq N-1$, the value function of player $i$ before arriving is 
$$
v_n:=\sup_{\lambda_i}E\left[R_{\xi(\tau^i)}-\int_0^{\tau^i} c_{\xi(s)}\lambda_i^2(\xi(s))ds \bigg| \xi(0)=n\right]
$$
where $\xi:=\xi_{\lambda_i,\lambda_{-i}}$. We also convene that $v_{N}:=R_{N}$.

\begin{remark}\label{rk:vN}
  The last player never arrives. Indeed, once $N-1$ players have arrived, the remaining player achieves the optimal value $v_{N-1}$ by using the control $\lambda_{i}\equiv0$, and in fact $v_{N-1}=R_{N}=v_{N}$. This is due to the convention that $R_{N}$ is paid to players that never arrive. On the other hand, this convention is necessary in order to have existence of an equilibrium, since the same value is asymptotically achieved by using the control $\lambda_{i}\equiv\eps$ with small $\eps>0$.
\end{remark}

\begin{proposition}\label{pr:Nplayer}
  The $N$-player game has a unique Nash equilibrium. The equilibrium value function $(v_{n})_{0\leq n\leq N}$ is the unique solution of the backward recursion
  \begin{equation}\label{eq:valueN}
  v_n=\frac{R_{n+1}+2(N-n-1)v_{n+1}}{1+2(N-n-1)}, \quad 0\leq n \leq N-1; \quad v_{N}=R_{N}.
  \end{equation}
  The unique equilibrium optimal control is
  \begin{equation}\label{eq:optconN}
    \lambda^*(n)=\frac{R_{n+1}-v_n}{2c_n}, \quad 0\leq n \leq N-1.
  \end{equation}
\end{proposition}

%%%%%%%%%%%%%%%%%%%%%%%%%%%%%%%%%%%
\subsection{The $N$-Player Principal--Agent Problem}

Next, we consider the $N$-player version of the mean field principal--agent problem introduced in Section~\ref{se:MFprincipal}. Given $n_0\in \{1,\ldots, N-1\}$ and a (nonnegative, decreasing) reward scheme $(R_{n})$, let
$$
  T_{n_0}=\inf\{t\ge 0:\xi_{\lambda^*}(t)= n_0\}
$$
be the (random) time until $n_0$ players have arrived, where the players use the unique equilibrium optimal control $\lambda^*$ for $(R_{n})$ and the (fixed, positive) cost coefficients~$(c_{n})$; cf.~\eqref{eq:optconN}. Below, we shall find it useful to write $\lambda^*_{n}$ rather than $\lambda^*(n)$ whenever we are in the $N$-player setting.

Given the per capita%
\footnote{A normalization of the budget is necessary for a convergence result as in the subsequent section. We have done that implicitly by seeing~$B$ as the total budget in the mean field limit and as per capita budget in the $N$-player setting. Equivalently, one could normalize the mass of the population by assigning mass $1/N$ to each agent in the $N$-player game and see $B$ as the total budget.
} reward budget $B>0$, the principal chooses a reward scheme $(R_n)$ such as to minimize the expected completion time $ET_{n_0}$, subject to the budget constraint $\sum_{n=1}^N R_n\le NB$. In analogy to~\eqref{eq:assumptCost}, we shall assume that $c^{N}$ satisfies
\begin{equation}\label{eq:assumptCostN}
  c^{N}_{n} \leq c^{N}_{n-1}\, \frac{(2N-2n+1)^{2}(2N-n-1)}{4(N-n-1)(N-n+1)(2N-n)} ,\quad n< n_0;
\end{equation}
again, this is satisfied e.g.\ when $n\mapsto c^{N}_{n}$ is constant or decreasing.

\begin{theorem}\label{thm:optRN}
  Let $c^{N}$ satisfy~\eqref{eq:assumptCostN} and define $y_{n_0}=0$,
  \[
    y_n=\sqrt{\frac{c_n N(N-n-1)}{(N-n)(2N-n-1)}}, \quad n<n_0.
  \]
  %Suppose $y_{n+1}\le \frac{1+2(N-n-1)}{2(N-n-1)} y_n$ for all $n<n_0$ (which holds e.g.\ for constant $c_n$). 
  Given a per capita reward budget $B>0$, there is a unique optimal reward scheme $(R^{*}_{n})$ attaining the minimal expected completion time $ET^*_{n_0}$, given by
  \[
    R^*_n=\frac{B}{C}\left\{y_{n-1}+\frac{1}{2}\sum_{k=n-1}^{n_0-1}\frac{ y_k}{N-k-1}\right\} \1_{\{n\le n_0\}},
  \]  
  and the minimal expected completion time is
  \[
    ET^*_{n_0}=\frac{4C^2}{B},\quad\mbox{where}\quad C=\frac{1}{2\sqrt{N}}\sum_{n=0}^{n_0-1} \sqrt{\frac{c_n(2N-n-1)}{(N-n)(N-n-1)}}.
  \]
  The corresponding equilibrium optimal control is
  \[
    \lambda^*_{n}= \frac{B}{2C}\sqrt{\frac{ N(N-n-1)}{c_n(N-n)(2N-n-1)}} \1_{\{n<n_0\}}.
  \]  
\end{theorem}

%%%%%%%%%%%%%%%%%%%%%%%%%
\subsection{Size Effects}\label{se:sizeEffects}

We conclude this section with a brief discussion of the influence of the population size $N$ on the principal's problem, in two different ways. To make the problems comparable, we assume that $c^{N}_n\equiv c$ is a constant independent of $N$.

(i) First, we consider as above a principal with a given per capita budget $B$ aiming to minimize the expected time until an $\alpha$ proportion of the population has arrived. The left panel in Figure~\ref{fig:size-effect-A} shows a negative size effect: the minimum expected completion time $ET^{N}_{n_{0}}$ for $n_0=\ceil{\alpha N}$ is increasing in $N$; that is, an increase in population size \emph{adversely} affects the principal.

(ii) Second, we fix the total budget $K=NB$ and a completion \emph{head count} $n_0$. Then,
$ET^{N}_{n_0}=\frac{1}{NK}\big(\sum_{n=0}^{n_0-1}\big[\frac{c}{1-n/N} (1+\frac{1}{1-(n+1)/N} )\big]^{1/2}\big)^2$
%\[ET^{N}_{n_0}=\frac{1}{NK}\left\{\sum_{n=0}^{n_0-1}\sqrt{\frac{c}{1-\frac{n}{N}} \left(1+\frac{1}{1-\frac{n+1}{N}}\right)}\right\}^2\]
is strictly decreasing in $N$ with limit equal to zero (right panel in Figure~\ref{fig:size-effect-A}). That is, the principal aiming for a fixed number of completions benefits from an increase in the population size. %In fact, if $n_0$ is fixed, then by doubling the size of the competition, the principal can half the total budget $K$ and still manage to obtain a similar expected completion time if $N\gg n_0$. 

\begin{figure}
\centering
\includegraphics[height=4.5cm]{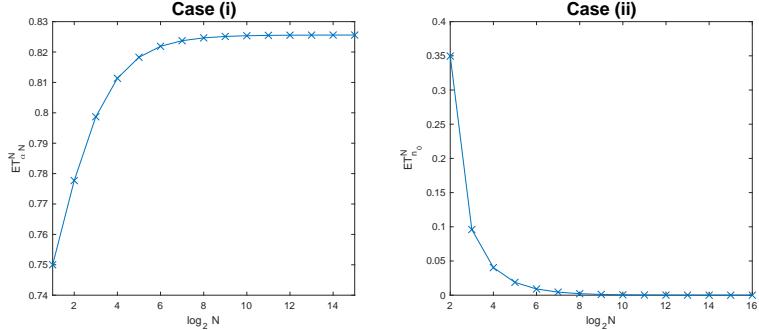}
\caption{The left panel shows a negative size effect when fixing $\alpha=0.5$ and $B=1$. The right panel shows a positive size effect when fixing $n_0=3$ and $K=2^{5}$.  In both panels, $c\equiv 1$ and the minimal expected completion time is plotted against $\log_2 N$.}
\label{fig:size-effect-A}
\end{figure}

%%%%%%%%%%%%%%%%%%%%%%%%%%%%%%%%%%%%%%%%%%%%%%%%
\section{Convergence to the Mean Field}\label{se:convergence}

In this section, we show that the $N$-player competition and principal--agent problems converge to their mean field counterparts as $N\to\infty$.

%%%%%%%%%%%%%%%%%%%%%%%%%%%%%%%%%%
\subsection{Convergence of the $N$-Player Equilibrium}

We consider the mean field setting of Section~\ref{se:equilib} with a fixed reward scheme~$R$ and cost coefficient $c$, as well as an $N$-player game with reward $(R^{N}_{n})$ and cost $(c^{N}_{n})$. Our aim is to show that if $R^{N}\to R$  and $c^{N}\to c$ in a suitable sense, then the corresponding equilibria converge.

If we start with a reward scheme $R:[0,1]\to \mathbb{R}_+$ in the mean field setting, an obvious choice for $R^{N}$ is the sampling $R^{N}_n=R\left(\frac{n}{N}\right)$. Since $R$ is decreasing, we have $\frac{1}{N}\sum_{n=1}^N R^{N}_n\le \int_0^1 R(r)\,dr$; i.e., this discretization may (and typically will) reduce the cumulative reward. Another choice is the moving average $R^{N}_n=N\int_{(n-1)/N}^{n/N} R(y)\,dy$ which preserves the reward. Next, we introduce a condition designed to cover either of these choices, and more.

Recall that $R:[0,1]\to \mathbb{R}_+$ is decreasing, piecewise Lipschitz continuous and left-continuous at $r=1$. Let $0=r_0<r_1<\ldots<r_m<r_{m+1}=1$ be a finite partition of $[0,1]$ such that $R$ is Lipschitz  on each interval $(r_{i-1},r_i)$. For our results, we shall assume that 
\begin{equation}\label{eq:discretizationR}
\sup_{r\in \bigcup_{i=1}^{m+1} I^{N}_i}\left|R^{N}_{\ceil{rN}}-R(r)\right|\le \frac{K}{N}
\end{equation}
for some constant $K$ independent of $N$, where $I^{N}_i=[r_{i-1}+1/N,r_i-1/N]$ for $i=1,\ldots,m$ and $I^{N}_{m+1}=[r_m+1/N,1]$. Similarly, we assume that%
\footnote{This covers, e.g., $c^{N}_{n}=c(n/N)$. We do not consider a more general convergence as in~\eqref{eq:discretizationR} since that would lead to more complicated statements below without substantially enhancing the scope.  
}
\begin{equation}\label{eq:discretizationc}
\sup_{r\in[0,1]}\left| c^{N}_{\ceil{rN}}-c(r)\right|\le \frac{K}{N}.
\end{equation}

The next result establishes that the $N$-player equilibrium converges to the mean field equilibrium as $N\rightarrow \infty$. Thus, it gives a second justification to the mean field formulation, apart from the direct derivation with a continuum of players as in Section~\ref{se:equilib}.

\begin{theorem} \label{thm:fwd-convergence}
  Let $R^{N},R$ and $c^{N},c$ satisfy~\eqref{eq:discretizationR} and~\eqref{eq:discretizationc}, and let $v^{N},v$ and $\lambda^{N},\lambda^{*}$ be the corresponding value functions and equilibrium optimal controls, respectively. Then
  \[
    \sup_{r\in[0,1]}|v^{N}_{\floor{rN}}-v(r)|=O(1/\sqrt{N}),\quad \sup_{r\in \bigcup_{i=1}^{m+1} I^{N}_i}|\lambda^{N}_{\floor{rN}}-\lambda^{*}(r)|=O(1/\sqrt{N}).
  \]
\end{theorem}

\begin{remark}\label{rmk:better-rate}
If $R(r)=0$ on $(\alpha,1]$ for some $\alpha<1$, then the ODE~\eqref{eq:HJODE} is nondegenerate up to the boundary and the convergence rates in Theorem~\ref{thm:fwd-convergence} can be improved to $O(1/N)$; cf.\ the proof of the theorem. In particular, this holds for the optimal reward scheme $R^*$ of the principal's problem in Theorem~\ref{th:principal}.
\end{remark}

%%%%%%%%%%%%%%%%%%%%%%%%%%%%%%
\subsection{Convergence and $\eps$-Optimality for the Principal}\label{subsec:eps-optimal-R}

Following Section~\ref{se:MFprincipal}, we fix a cost coefficient $c$ satisfying~\eqref{eq:assumptCost}, a target proportion $\alpha\in(0,1)$ and a budget $B>0$. We have seen in Theorem~\ref{th:principal} that in the mean field setting, there exists a unique reward scheme~$R^{*}$ which attains the minimal (deterministic) time $T^{*}_{\alpha}$ until an $\alpha$-proportion of the population has arrived. For the $N$-player situation with a given cost $c^{N}$ satisfying~\eqref{eq:assumptCostN} and per capita
budget $B$, we have seen in Theorem~\ref{thm:optRN} that there exists a unique reward scheme $R^{N}$ minimizing the expected time $ET^{N}_{n_{0}}$ until $n_{0}$ players have arrived. In the following result, we consider $n_{0}=\ceil{\alpha N}$, so that the proportion $n_{0}/N$ tends to $\alpha$, and show that if $c^{N}\to c$, the expected completion time as well as the corresponding reward schemes converge as $N\to\infty$.

\begin{theorem}\label{thm:principal-convergence}
  Let $c^{N}\to c$ in the sense of~\eqref{eq:discretizationc}. Then 
  \[
    \left|ET^{N}_{\ceil{\alpha N}} - T^*_{\alpha} \right|= O(1/N),\qquad \sup_{r\in[0,\alpha]}\left| R^{N}_{\ceil{r N}}-R^*(r)\right| = O(1/N).
  \]
\end{theorem}

The next result addresses the convergence of the principal's problem in a different sense: it shows that if the principal applies (a discretization of) the optimal reward scheme $R^{*}$ from the mean field setting in an $N$-player game with large $N$, rather than the precise optimal scheme $R^{N}$ of Theorem~\ref{thm:optRN}, then $R^{*}$ is still $\eps$-optimal for the minimization of the expected completion time.

\begin{corollary}\label{co:epsOptimalPrincipal}
  Let $c^{N}\to c$ in the sense of \eqref{eq:discretizationc} and let $R^{(N)}$ be a discretization of $R^{*}$ satisfying~\eqref{eq:discretizationR}. Then the completion time $T^{(N)}_{\ceil{\alpha N}}$ of the $N$-player game under $R^{(N)}$ satisfies 
  \[
    \left|ET^{(N)}_{\ceil{\alpha N}} - ET^{N}_{\ceil{\alpha N}}\right| = O(1/N);
  \]
  that is, the reward scheme $R^{(N)}$ is $O(1/N)$-optimal for the $N$-player principal--agent problem.
\end{corollary}

%%%%%%%%%%%%%%%%%%%%%%%%%%%%%%%%%%%%%%%%%%%%%%%%%%%%%%%%%%%%%%%%%%
\appendix
%%%%%%%%%%%%%%%%%%%%%%%%%%%%%%%%%%%%%%%%%%%%%
\section{Proofs}\label{se:proofs}

%%%%%%%%%%%%%%%%%%%%%%%%%%%%%%%%%%%%%%%%%%%%%
\subsection{Proofs for Section~\ref{se:equilib}}

\begin{proof}[Proof of Lemma~\ref{le:kolmogorovEqn}.]
  The piecewise Lipschitz property of $\lambda$ implies that~\eqref{eq:kolmogorov} has a unique continuous solution $\rho$. This function is nonnegative, increasing, globally Lipschitz continuous, and continuously differentiable except at finitely many points (corresponding to the jumps to $\lambda$) where the left and right derivatives of $\rho$ may disagree.
  Let $\bar\rho(t) = 1-\exp(-\int_{0}^{t} \lambda(\rho(s))\,ds)$.
  By the Exact Law of Large Numbers (Proposition~\ref{pr:ELLN}) and~\eqref{eq:arrivalTimeDefn},
  $$
    \mu\{i:\, \tau_{\lambda}^{i}(\omega)\in [0,t]\} = P\{\omega:\, \tau_{\lambda}^{i}(\omega)\in [0,t]\}=\bar\rho(t)
  $$
  holds almost-surely. On the other hand, the derivative of $\bar\rho$ is seen to satisfy
  $\bar\rho'(t) = \lambda(\rho(t))(1-\bar\rho(t))$ a.e.,
  so that the uniqueness of the exponential ODE yields $\bar\rho=\rho$.
%  
%  Replacing $\lambda(t)$ with its right limit $\lambda(t+)$ does not change the equation~\eqref{eq:kolmogorov}, so we may suppose that $\lambda$ is right continuous. In view of
%  \begin{align*}
%    \frac{\bar\rho(t+h)-\bar\rho(t)}{1-\bar\rho(t)}
%    & = \exp\left(-\int_{t}^{t+h} \lambda(\rho(s))\,ds\right),\quad h>0,
%  \end{align*}
%  dividing by $h$ and letting $h$ tend to zero then yields that 
%  $$
%    \frac{\bar\rho'(t)}{1-\bar\rho(t)} = \lambda(\rho(t)),
%  $$
%  where $\bar\rho'(t)$ is the right derivative and $h^{-1} \int_{t}^{t+h} \lambda(\rho(s))\,ds\to \lambda(\rho(t))$ was used. Now, the uniqueness of~\eqref{eq:kolmogorov} implies that $\bar\rho=\rho$.
\end{proof}

\begin{proof}[Proof of Theorem~\ref{th:equilibrium}.]
  (i) Suppose $\bar\lambda\in\Lambda$ is an equilibrium optimal control, let $\rho$ be the corresponding equilibrium state process and let $\bar{v}$ be the corresponding value function of~\eqref{eq:valueSinglePlayer} for any given player before arrival. Using the constant control $\lambda\equiv0$ shows that $R(1) \leq \bar{v}(r) \leq R(r)$ and hence $\bar{v}(1-)=R(1)$; recall that $R$ is left-continuous at $r=1$. Let 
  $$
    r_{0}=\inf\{r:\, R(r)=R(1)\}.
  $$
  Using $\lambda\equiv0$ also shows that $\bar{v}\equiv R(1)$ on $[r_{0},1]$, whereas 
  $\bar{v}< R$ on $[0,r_{0})$ because the only control with zero cost merely attains $R(1)<R(r)$.   The inequality $R(1)<R(r)$ on $[0,r_{0})$ also implies that $\bar\lambda>0$ a.e.\ on $[0,r_{0})$ 
  % if $\lambda(r-)=0$, the state $\rho$ would stop moving and that would yield $v=R(1)$
  and in fact, recalling that $\bar\lambda\in\Lambda$, even that $\bar\lambda$ is a.e.\ uniformly bounded away from zero on any interval $[0,r_{1}]$ with $r_{1}<r_{0}$. 
  
  A control-theoretic argument shows that $\bar{v}$ is Lipschitz on any such interval. Indeed, let $0\leq r < r_{1}$ and choose $h>0$ such that $r+2h<r_{1}$. We may compare the optimal control $\bar{\lambda}$ started at $r$ with a control $\lambda$ such that 
  %$\lambda=0$ on $[r,r+h)$, $\lambda=a\bar{\lambda}$ on $[r+h,r+2h)$, and $\lambda=\bar{\lambda}$ on $[r+2h,1]$, where the constant $a\geq1$ is chosen such that the integral of $\lambda$ over $[r,r+2h]$ coincides with the integral of $\bar\lambda$ over the same interval. 
  $$
  \lambda = \begin{cases}
  0 & \mbox{ on } [r,r+h),\\
  a\bar{\lambda} & \mbox{ on } [r+h,r+2h),\\
  \bar{\lambda} & \mbox{ on } [r+2h,1),\\
  \end{cases}
  $$
  where the constant $a\geq1$ is chosen such that the integral of $\lambda$ along $\rho$ over $[r,r+2h]$ coincides with the integral of $\bar\lambda$ over the same interval. Then, both the extra cost of $\lambda$ and the loss of expected reward are bounded by a  constant (depending only on $r_{1}$) times $h$, because if $\bar\lambda$ achieves a rank in $[r+2h,1)$, then $\lambda$ achieves the same rank, whereas the probability of $\bar\lambda$ ranking in $[r,r+2h)$ is bounded by a constant times $h$. Since $\lambda$ is an admissible control for the control problem started at $r+h$, it follows that $0\leq \bar{v}(r)-\bar{v}(r+h)\leq Ch$ as claimed. In particular, $\bar{v}$ is absolutely continuous and a.e.\ differentiable.
  
  By dynamic programming, the value function $\bar{v}$ must then a.e.\ satisfy the Hamilton--Jacobi equation
  $$%\begin{equation}\label{eq:HJeqn}
        \sup_{l\geq0} \{l[R(r)-v(r)]-c(r)l^2\}+\bar\lambda(r)(1-r) v'(r)=0\mbox{ on }[0,1), \quad v(1)=R(1).
  $$%\end{equation}
  Moreover, the optimal control $\bar\lambda$ must attain the supremum a.e. Recalling that $\bar{v}= R$ on $[r_{0},1]$ and $\bar{v}< R$ on $[0,r_{0})$, it follows that
  \begin{equation}\label{eq:firstOrderControl}
    \bar\lambda(r)=\frac{R(r)-\bar{v}(r)}{2c(r)} \quad\mbox{a.e.}
  \end{equation}
  and that $\bar{v}$ satisfies
%  \[
%    \frac{[R(r)-v(r)]^2}{4c(r)}+\frac{R(r)-v(r)}{2c(r)}(1-r) v'(r)=0\mbox{ on }[0,1), \quad v(1)=R(1),
%  \]  
%  and recalling that $\bar{v}= R$ on $[r_{0},1]$ whereas $\bar{v}< R$ on $[0,r_{0})$, we have that
%  $\bar{v}$ satisfies
  \begin{equation}\label{eq:HJODE}
    R(r)-v(r)+2(1-r) v'(r)=0\mbox{ a.e.\ on }[0,r_{0}), \quad v\equiv R(1)\mbox{ on }  [r_{0},1].
  \end{equation}
  
  (ii)  Using the regularity of $R$, we see that the function $v$ of~\eqref{eq:value} is Lipschitz on $[0,1)$ and satisfies $v(1-)=R(1)$. We extend $v$ to a Lipschitz function on $[0,1]$ by setting $v(1)=R(1)$. A direct calculation shows that $v$ satisfies~\eqref{eq:HJODE} and that $\lambda^{*}$ of~\eqref{eq:equilibOptControl} is the corresponding maximizer~\eqref{eq:firstOrderControl}. Moreover, $\lambda^{*}\in\Lambda$. A verification argument then yields that $v$ is the value function and $\lambda^{*}$ is an optimal control.
  
  (iii) It remains to show that~\eqref{eq:HJODE} has at most one absolutely continuous solution. Indeed, if $v_{1}$ and $v_{2}$ are solutions, then $w=v_{1}-v_{2}$ is absolutely continuous and satisfies
  $$%\begin{equation}\label{eq:HJODEdiff}
     2(1-r) w'(r)=w(r)\mbox{ a.e.\ on }[0,r_{0}), \quad w\equiv 0\mbox{ on }  [r_{0},1].
  $$%\end{equation}  
  If $r_{0}<1$, this is a Lipschitz ODE and it follows directly that $w\equiv 0$ is the unique solution. If $r_{0}=1$, we set $u(t)=w(1-e^{-2t})$ for $t\geq0$; then $u$ satisfies $u'(t)=u(t)$ on $[0,\infty)$ and hence $u(t)=u(0)e^{t}$. But $u(\infty)=w(1)=0$ then yields that $u\equiv0$ and thus $w\equiv 0$ as desired.
\end{proof}

\begin{proof}[Proof of Proposition~\ref{pr:valueProbabRepresent}.]
  Let $V(r)= E[R(\rho(\tau))|\rho(0)=r]$. The ODE for $\rho$ has the unique solution $\rho(t)=1-(1-r) e^{-2\lambda_{0}t}$ for $t\geq0$. Hence,
  $$
    V(r)=E[R(1-(1-r) e^{-2\lambda_{0}\tau})]=\int_{0}^{\infty} \lambda_{0}e^{-\lambda_{0}x}R(1-(1-r) e^{-2\lambda_{0}x})\,dx.
  $$
  A change of variables then shows that $V(r)$ coincides with~\eqref{eq:value}.
\end{proof}

\begin{proof}[Proof of Proposition~\ref{pr:stability}.]
  Note that $R_{n},R$ have a uniform upper bound given by $\sup_{n} R_{n}(0)$. Moreover, by monotonicity, the convergence $R_{n}\to R$ is uniform on each interval of continuity of $R$. It follows directly from~\eqref{eq:value} that the value functions $v_{n}$ converge uniformly to their counterpart~$v$. Similarly, \eqref{eq:equilibOptControl} yields that the optimal controls $\lambda^{*}_{n}$ converge pointwise to their counterpart $\lambda^{*}$, and uniformly on each interval of Lipschitz continuity of $R$. Moreover, there is a uniform upper bound for the sequence $(\lambda^{*}_{n})$. By the ODE~\eqref{eq:kolmogorov}, this entails an upper bound for the Lipschitz constants of $(\rho_{n})$. Thus, after passing to a subsequence, $(\rho_{n})$ converges uniformly to a limit~$\bar{\rho}$. To verify that $\bar\rho=\rho$, it suffices to show that $\bar\rho$ solves the ODE~\eqref{eq:kolmogorov} defining $\rho$ on each of the mentioned intervals, and that follows from the uniform convergence of $(\lambda^{*}_{n})$. Finally, by a subsequence argument, the entire sequence $(\rho_{n})$ must converge to~$\rho$.
\end{proof}
%%%%%%%%%%%%%%%%%%%%%%%%%%%%%%%%%%%%%%%%%%%%%
\subsection{Proofs for Section~\ref{se:MFprincipal}}

As a preparation for the proof of Theorem~\ref{th:principal}, we first show that no reward should be distributed to ranks below $\alpha$---this is quite intuitive since the planner does not care about agents arriving after rank $\alpha$. The converse is also true.

\begin{lemma} \label{le:cutoff}
  The value $T^{*}_{\alpha}$ of~\eqref{eq:optimalContractValue} does not change if the infimum is restricted to $R\in\cR$ satisfying $R >0$ on $[0,\alpha)$ and $R =0$ on $(\alpha,1]$.
%  $$
%    R >0 \mbox{ on }[0,\alpha)\quad\mbox{and}\quad R =0 \mbox{ on }(\alpha,1].
%  $$
\end{lemma}

\begin{proof}
  For the first property, suppose that $R\in\cR$ vanishes at $r\in[0,\alpha)$, and hence on $[r,1]$. Then the equilibrium effort $\lambda$ of~\eqref{eq:equilibOptControl} also vanishes at~$r$, which by the Kolmogorov equation~\eqref{eq:kolmogorov} implies that the state $\rho(t)$ never exceeds $r$, and hence that $T_{\alpha}(R)=\infty$.
  
  To see the second property, let $R\in\cR$ and set $\hat R=R\1_{[0,\alpha]}$; then $\hat R\in\cR$. For $r\in [0,\alpha]$, the corresponding equilibrium efforts $\lambda$ and $\hat\lambda$ of~\eqref{eq:equilibOptControl} satisfy
  \begin{align*}
    \hat\lambda(r)
      =\frac{\hat R(r)-\frac{1}{2\sqrt{1-r}}\int_r^1 \frac{\hat R(y)}{\sqrt{1-y}}\,dy}  {2c(r)} 
      %=\frac{R(r)-\frac{1}{2\sqrt{1-r}}\int_r^\alpha \frac{R(y)}{\sqrt{1-y}}\,dy}{2c(r)}\\
      \geq\frac{R(r)-\frac{1}{2\sqrt{1-r}}\int_r^1 \frac{R(y)}{\sqrt{1-y}}\,dy}{2c(r)}=\lambda(r),
  \end{align*}
  and the inequality is strict if $\int_{\alpha}^{1}\frac{R(y)}{\sqrt{1-y}}\,dy>0$. As a result, if $R$ does not vanish on $(\alpha,1]$, then $\hat R$ produces a strictly larger equilibrium effort on $[0,\alpha]$, and hence  $T_{\alpha}(\hat R)<T_{\alpha}(R)$ whenever $T_{\alpha}(R)<\infty$, by~\eqref{eq:kolmogorov}.
\end{proof}

We can now show the theorem through a calculus of variations argument.

\begin{proof}[Proof of Theorem~\ref{th:principal}.]
  %We set $T(R)=T_{\alpha}(R)$ for brevity. 
Let 
$$
  \cR'=\big\{R\in\cR:\, \smallint R(r)\,dr\leq B,\,T_{\alpha}(R)<\infty, \, R\1_{(\alpha,1]}=0\big\}.
$$
As noted, $\cR'\neq\emptyset$ by~\eqref{eq:quantileUniformCutoff}, and by Lemma~\ref{le:cutoff} it is sufficient to show that $R^{*}$ is the unique optimizer in $\cR'$. Moreover, we have seen in the proof of Lemma~\ref{le:cutoff} that $\lambda^{*}$ is a.e.\ strictly positive on $[0,\alpha)$ for $R\in\cR'$. Hence, $\rho$ is strictly increasing and we have $T_{r}(R)=\rho^{-1}(r)$ for $r\in [0,\alpha)$. Differentiating $\rho^{-1}(r)$ and using the ODE~\eqref{eq:kolmogorov} for $\rho$ and~\eqref{eq:equilibOptControl}, we then obtain that
\begin{align}
  T_{\alpha}(R) 
  &=\int_0^\alpha \frac{1}{(1-r)\lambda^{*}(r)}\,dr \nonumber \\
  &= \int_0^\alpha \frac{2c(r)}{\sqrt{1-r}\left(R(r)\sqrt{1-r}-\int_r^\alpha \frac{R(s)}{2\sqrt{1-s}}\,ds\right)}\,dr,\quad R\in\cR'. \label{eq:Talpha}
\end{align}
We see from~\eqref{eq:Talpha} that $R\mapsto T_{\alpha}(R)$ is strictly convex on $\cR'$, up to a.e.\ equivalence. This implies that there is at most one optimal $R\in\cR'$.

Next, we derive a sufficient condition for optimality. We first reparametrize the optimization problem: for $R\in\cR'$, we consider
\[
  f(r)=f_{R}(r)=R(r)\sqrt{1-r}-\int_r^\alpha \frac{R(s)}{2\sqrt{1-s}}\,ds,\quad r\in[0,\alpha].
\]
The mapping $R\mapsto f_{R}$ is one-to-one on $\cR'$ since $R$ can be recovered from $f$ via
\begin{equation}\label{eq:f-to-R}
  R(r) = \frac{f(r)}{\sqrt{1-r}}+\int_r^\alpha \frac{f(s)}{2(1-s)^{3/2}}\,ds.
\end{equation}
Indeed, if $R$ is differentiable, then $f'(r)=\sqrt{1-r}R'(r)$ and integration by parts yields
\begin{align*}
  R(r) &=R(\alpha)-\int_r^\alpha \frac{f'(s)}{\sqrt{1-s}}\,ds \\
   &=\frac{f(\alpha)}{\sqrt{1-\alpha}}-\int_r^\alpha \frac{f'(s)}{\sqrt{1-s}}\,ds 
   =\frac{f(r)}{\sqrt{1-r}}+\int_r^\alpha \frac{f(s)}{2(1-s)^{3/2}}\,ds,
\end{align*}
and now the claim for general $R\in\cR'$ follows by approximation. Fubini's Theorem shows that $\int_{0}^{\alpha} R(r)\,dr = \frac12 \int_0^\alpha \frac{(2-r)f(r)}{(1-r)^{3/2}}\,dr$. Thus, recalling that $\alpha<1$, the image $\cF$ of $\cR'$ under $R\mapsto f_{R}$ is the convex set of all piecewise Lipschitz, nonnegative, decreasing
%\footnote{\YZ{$R$ is decreasing if and only if $f$ is decreasing: on each interval of Lipschitz continuity of $R$, $f'(r)=\sqrt{1-r}R'(r)$; at each jump point of $R$, $\Delta f(r)=\sqrt{1-r}\Delta R(r)$.}} 
functions $f:[0,\alpha]\to \R$ such that~\eqref{eq:Tf} is finite and the budget constraint
\begin{equation}\label{eq:constraint-f}
  \frac12 \int_0^\alpha \frac{(2-r)f(r)}{(1-r)^{3/2}}\,dr\leq B
\end{equation}
is satisfied.
We write $T_\alpha(f_{R})$ for $T_\alpha(R)$ by a slight abuse of notation; then by~\eqref{eq:Talpha} we have
\begin{equation}\label{eq:Tf}
   T_\alpha(f)= \int_0^\alpha \frac{2c(r)}{\sqrt{1-r} f(r)}\,dr,\quad f\in \cF.
\end{equation}
The mapping $f\mapsto T_\alpha(f)$ 
is convex and finite-valued, and clearly, $f^{*}\in\cF$ is optimal if an only if
$$
  \eps \mapsto \phi(\eps)=T((1-\eps)f^{*}+ \eps f),\quad [0,1]\to\R
$$
attains a minimum at $\eps=0$ for all $f\in\cF$. By convexity, this function has a right derivative $\phi'(0)$ at $\eps=0$, and $\phi$ attains a minimum at $\eps=0$ if and only if $\phi'(0)\geq0$. Note that for any convex function $\varphi$, the right difference quotient $(\varphi(x+\eps)-\varphi(x))/\eps$ satisfies $\varphi'(x) \leq (\varphi(x+\eps)-\varphi(x))/\eps \leq \varphi(x+1)-\varphi(x)$ for $\eps\le 1$. Using these bounds and dominated convergence, we see that $\phi'(0)$ can be computed by differentiating under the integral:
\begin{equation}\label{eq:auxDerivative}
  \phi'(0) = \int_0^\alpha \frac{-2c(r)(f(r)-f^{*}(r)) }{\sqrt{1-r}f^{*}(r)^{2}}\,dr.
\end{equation}
Let $R^{*}$ be as in~\eqref{eq:optimalR}; then the corresponding function $f^{*}=f_{R^{*}}$ is given by
\begin{equation}\label{eq:f*}
  f^*(r)=\frac{B}{C}\sqrt{\frac{c(r)(1-r)}{2-r}}, \quad r\in[0,\alpha].
\end{equation}
Recalling that $\alpha<1$ and that $\frac{c(r)(1-r)}{2-r}$ is decreasing, we verify directly that $f^*\in \cF$. Moreover, the budget constraint~\eqref{eq:constraint-f} is satisfied with equality.

Fix an arbitrary $f\in\cF$. Using the expression~\eqref{eq:f*} in the denominator of~\eqref{eq:auxDerivative}, we have
$$
  \phi'(0) = \frac{-2C^{2}}{B^{2}}\int_0^\alpha \frac{(2-r)(f(r)-f^{*}(r))}{(1-r)^{3/2}}\,dr.
$$
Since $f^{*}$ satisfies~\eqref{eq:constraint-f} with equality and $f$ satisfies the same with inequality, the above integral is nonpositive.
%$$
%  \int_0^\alpha \frac{(2-r)(f(r)-f^{*}(r))}{(1-r)^{3/2}}\,dr\leq 0.
%$$
Thus, $\phi'(0)\geq0$, showing that $f^{*}\in\cF$ is optimal and hence that $R^{*}\in\cR$ is optimal. The formula~\eqref{eq:optimalTime} for $T^{*}_{\alpha}$ is then obtained from~\eqref{eq:Tf} and~\eqref{eq:f*}, and the formula for $\lambda^{*}$ follows via~\eqref{eq:equilibOptControl}.
\end{proof}

%%%%%%%%%%%%%%%%%%%%%%%%%%%%%%%%%%%%%%%%%%%%%
\subsection{Proofs for Section~\ref{se:NplayerProblems}}

\begin{proof}[Proof of Proposition~\ref{pr:Nplayer}.]
  Fix player $i$ and suppose that all other players use a control $\lambda_{-i}$. By dynamic programming, the value function $v_{n}$ of player $i$ before arrival satisfies
  $$
  \sup_{\lambda_i\ge 0}\left\{\lambda_i [R_{n+1}-v_n]-c_n\lambda_i^2\right\}+\lambda_{-i}(n) (N-n-1)(v_{n+1}-v_n)=0
  $$
  for $0\leq n \leq N-2$. For $n=N-1$, the same holds by Remark~\ref{rk:vN} and our convention that $v_{N}=R_{N}$. Thus, \eqref{eq:optconN} is the optimal control for player $i$ 
  and it follows that
  \[
    \frac{(R_{n+1}-v_n)^2}{4c_n}+\lambda_{-i}(n) (N-n-1)(v_{n+1}-v_n)=0.
  \]
  Assuming inductively that $v_{n+1}\leq R_{n+1}$, this quadratic equation has a unique nonnegative root $v_{n}$, and $v_{n}$ satisfies $0\leq v_{n} \leq R_{n+1} \leq R_{n}$.
  In a given equilibrium, the consistency condition $\lambda_i=\lambda_{-i}$ implies that  
  \begin{equation}\label{eq:differenceEq}
  R_{n+1}-v_n+2(N-n-1)(v_{n+1}-v_n)=0, \quad n=0,\ldots, N-1,
  \end{equation}
  or equivalently~\eqref{eq:valueN}, which clearly has a unique solution. Conversely, we can verify directly that~\eqref{eq:valueN} and~\eqref{eq:optconN} define an equilibrium.
\end{proof}

\begin{proof}[Proof of Theorem~\ref{thm:optRN}.]
  We first observe that $T_{n_0}$ is a sum of independent exponential random variables (whenever finite) and thus
  \begin{equation}\label{eq:ET}
  E T_{n_0}=\sum_{n=0}^{n_0-1}\frac{1}{(N-n)\lambda_{n}}.
  \end{equation}
  Moreover, similarly as in Lemma~\ref{le:cutoff}, it suffices to consider reward schemes with $R_{n}>0$ for $n\leq n_{0}$ and $R_{n}=0$ for $n> n_{0}$. Indeed, the first claim is immediate from~\eqref{eq:optconN}. To obtain the second, we argue by contradiction and compare $(R_{n})$ with the scheme defined by $\hat{R}_n=R_n\1_{\{n\le n_0\}}$. Proposition~\ref{pr:Nplayer} implies that $\hat{R}_n$ leads to a strictly larger equilibrium control before $n_{0}$ and hence a strictly smaller completion time. Thus, we only consider reward schemes up to $n=n_{0}$ in what follows.
%  DETAILS:
%  \begin{proposition}\label{prop:Tn}
%  The optimal non-negative $N$-player reward scheme to minimize $T_{n_0}$ exists, is unique, assigns positive reward to the first $n_0$ places and zero reward after the $n_0$-th place.
%  \end{proposition}
%  \begin{proof}
%  We first show that it is not optimal to distribute any reward after the $n_0$-th place. To this end, let $R:\mathcal{S}_N\rightarrow \mathbb{R}_+$ be a reward scheme such that $R_{n_1}>0$ for some ${n_1}>n_0$.
%  Set $\hat{R}_n=R_n\1_{\{n\le n_0\}}$. Also let $\lambda$ and $v$ (resp.\ $\hat\lambda$ and $\hat v$) be the equilibrium optimal control and equilibrium value function corresponding to $R$ (resp.\ $\hat R$), respectively. We have $\hat v_{n_0}=0< v_{n_0}.$ Suppose $\hat v_{n+1}< v_{n+1}$ with $ n\le n_0-1$, then
%  \begin{align*}
%  \hat v_n=\frac{\hat R_{n+1}+2(N-n-1)\hat v_{n+1}}{1+2(N-n-1)}< \frac{R_{n+1}+2(N-n-1)v_{n+1}}{1+2(N-n-1)}=v_n
%  \end{align*}
%  By induction, we conclude that $\hat v_n< v_n$ for all $n\le n_0$. It follows that
%  \[\hat \lambda_n=\frac{\hat R_{n+1}-\hat v_n}{2c_n}> \frac{R_{n+1}-v_n}{2c_n}=\lambda_n, \quad n=0,\ldots, n_0-1.\]
%  and thus $ET_{n_0}(\hat R)<ET_{n_0}(R)$.
  
  From \eqref{eq:valueN}, \eqref{eq:optconN} and \eqref{eq:ET}, we see that $ET_{n_0}: \mathbb{R}^{n_0}_{+} \rightarrow [0,\infty]$ is a strictly convex, continuous function of $(R_1,\ldots, R_{n_0})$. Moreover, the feasible set defined by $R_1\ge R_2\ge \cdots \ge R_{n_0}\ge 0$ and $\sum_{n=1}^N R_n\le NB$ is nonempty, convex and compact. As a result, there exists a unique optimal reward scheme.
  In the remainder of the proof, we determine this reward scheme explicitly. To that end, define $x_{n_0}=0$ and $x_n=R_{n+1}-v_n$ for $n< n_0$. By~\eqref{eq:optconN} and~\eqref{eq:ET}, the objective function can then be expressed as
  \begin{align*}
  ET_{n_0}=\sum_{n=0}^{n_0-1}\frac{2c_n}{(N-n) x_n}.
  \end{align*}
  From \eqref{eq:valueN} we obtain that
%  \[x_n=\frac{2(N-n-1)}{1+2(N-n-1)}(x_{n+1}+R_{n+1}-R_{n+2})\]
%  or
  $
    R_{n+1}-R_{n+2}=\frac{1+2(N-n-1)}{2(N-n-1)}x_n-x_{n+1}
  $
  and thus the total reward $\sum_{n=1}^{n_0} R_n=\sum_{n=1}^{n_0} n(R_{n}-R_{n+1})$ can be expressed as 
  \begin{align*}
  \sum_{n=1}^{n_0} R_n%&=\sum_{n=1}^{n_0} n(R_{n}-R_{n+1})\\
  &=\sum_{n=1}^{n_0} n\left\{\frac{1+2(N-n)}{2(N-n)}x_{n-1}-x_{n}\right\}=\sum_{n=0}^{n_0-1} \frac{2N-n-1}{2(N-n-1)} x_n.
  \end{align*}
  The constrained optimization problem of finding $x_{n}\in\R$ which
  \[
    \text{minimize}\quad \sum_{n=0}^{n_0-1}\frac{2c_n}{(N-n) x_n} \quad\text{subject to}\quad \sum_{n=0}^{n_0-1} \frac{2N-n-1}{2(N-n-1)} x_n\le NB
  \]
  can be solved using the method of Lagrange multipliers. One finds that the optimal $x_n$ are given by
  \[
    x_n=2\sqrt{\frac{c_n(N-n-1)}{\theta(N-n)(2N-n-1)}}, \quad n=0, \ldots, n_0-1
  \]
  where the Lagrange multiplier $\theta$ satisfies
  \[
    \sqrt{\theta}=\frac{1}{NB}\sum_{n=0}^{n_0-1} \sqrt{\frac{c_n (2N-n-1)}{(N-n)(N-n-1)}}.
  \]
  Note that $\theta$ and $x_n$ are related to $C$ and $y_n$ of Theorem~\ref{thm:optRN} by $C=B\sqrt{N\theta}/2$ and $x_n=2 y_n/\sqrt{N\theta}=By_n/C$. We have that $R_{n_0}\ge 0$ as $x_{n_0-1}\ge 0$, and 
  $R_{n}-R_{n+1}=\frac{B}{C} \big(\frac{1+2(N-n)}{2(N-n)}y_{n-1}-y_{n} \big)\ge 0$ for $n<n_0$; here the last inequality is equivalent to~\eqref{eq:assumptCostN}.
  Thus, the reward scheme $(R^*_{n})$ associated with the optimizer $(x_{n})$ is indeed nonnegative and decreasing, and it follows that $(R^*_{n})$ is the optimal reward scheme. The formulas for $R^*$, $T^*_{n_0}$ and $\lambda^*$ follow by direct calculation.
\end{proof}

%%%%%%%%%%%%%%%%%%%%%%%%%%%%%%%%%%%%%%%%%%%%%
\subsection{Proofs for Section~\ref{se:convergence}}

\begin{proof}[Proof of Theorem~\ref{thm:fwd-convergence}.]
Let $N\geq 4$ be large enough such that $\delta:=1/N$ satisfies $r_m<1-\sqrt{\delta}-\delta$. This implies, in particular, that $1-r_m>\delta+\sqrt{\delta}> 2\delta$. We may rewrite the recursion~\eqref{eq:valueN} for $v^{N}_{n}$ as
\begin{equation}\label{eq:vn}
v^{N}_{n}=g\left(\frac{n}{N}\right) v^{N}_{n+1}+ \left(1-g\left(\frac{n}{N}\right)\right)R^{N}_{n+1}
\end{equation}
where
\[g(r)=\frac{2\left(1-r-\delta\right)}{\delta+2\left(1-r-\delta\right)}.\]
Using~\eqref{eq:value}, we write the mean field value in a similar form:
\begin{equation}\label{eq:vr}
v\left(\frac{n}{N}\right)=f\left(\frac{n}{N}\right)v\left(\frac{n+1}{N}\right)+\left(1-f\left(\frac{n}{N}\right)\right) R^{N}_{n+1}+E_{1,n}
\end{equation}
where
\[
  f(r):=\sqrt{\frac{1-r-\delta}{1-r}} \quad\mbox{and}\quad E_{1,n} :=\frac{1}{2\sqrt{1-n\delta}}\int_{n\delta}^{(n+1)\delta}\frac{R(y)-R^{N}_{\ceil{yN}}}{\sqrt{1-y}}\,dy.
\]
%\begin{align*}
%E_1^{N}(r)&:=\frac{1}{2\sqrt{1-r}}\int_r^{r+\delta}\frac{R(y)-R^{N}_{\ceil{yN}}}{\sqrt{1-y}}dy.
%\end{align*}
Next, we estimate $\Delta_{n}:=\left|v^{N}_{n}-v\left(\frac{n}{N}\right)\right|$. For the last rank, we have $\Delta_N=\left|R^N_N-R(1)\right|\le K\delta$ by \eqref{eq:discretizationR}.
%As $v^{N}_{N}=R(1)=v(1)$, we have $\Delta_N=0$. 
For the second-to-last rank, since $1-\delta\ge 1-\sqrt{\delta}>r_m+\delta$, \eqref{eq:discretizationR} again implies that
%\[\Delta_N=\left| N\int_{1-\frac{1}{N}}^1 R(y)dy-R(1)\right|\le N\int_{1-\frac{1}{N}}^1 \left|R(y)-R(1)\right|dy\le \frac{K}{N}.\]
%and
\begin{align*}
\Delta_{N-1}&=\left|R^{N}_N-v\left(1-\delta\right)\right|\le \frac{1}{2\sqrt{\delta}}\int_{1-\delta}^1\frac{|R^{N}_{\ceil{yN}}-R(y)|}{\sqrt{1-y}}\,dy\le K\delta.
\end{align*}
For $n\le N-2$, subtracting \eqref{eq:vr} from \eqref{eq:vn} and using that $g(r)\le 1$ when $0\le r\le 1-\delta$, we obtain that 
\begin{equation}\label{eq:fconvDelta}
\Delta_n
\le
\Delta_{n+1} +E_{1,n}+E_{2,n},
\end{equation}
where
\[E_{2,n}=\left|g\left(\frac{n}{N}\right)-f\left(\frac{n}{N}\right)\right|\cdot \left|v\left(\frac{n+1}{N}\right)-R^{N}_{n+1}\right|.\]
Next, we estimate $E_{1,n}$ and $E_{2,n}$. To that end, it will be useful to define $n_0:=N-\ceil{\sqrt{N}}$ and note that $n\le n_0$ if and only if $n\delta \le 1-\sqrt{\delta}$.

\vspace{.5em}

\emph{Estimation of $E_{1,n}$.}
If $n_0+1\le n\le N-2$, then $r_m+\delta<n\delta\le 1-2\delta$, and \eqref{eq:discretizationR} implies
\begin{equation*}
E_{1,n}\le K\delta \left(1-\sqrt{\frac{1-n\delta-\delta}{1-n\delta}}\right)\le K\delta.
\end{equation*}
On the other hand, if $0\le n\le n_0$ and $\left[n\delta,(n+1)\delta\right]\subseteq I^{N}_i $ for some $i$, then \eqref{eq:discretizationR} and $1-n\delta\ge \sqrt{\delta}$ imply
\[E_{1,n}\le K\delta\frac{\frac{\delta}{1-n\delta}}{1+\sqrt{\frac{1-n\delta-\delta}{1-n\delta}}}\le K\delta^{3/2}.\]
We observe that there are at most $3m$ indices $n$ for which $n\le n_0$ and $\left[n\delta,(n+1)\delta\right]$ is not contained in any $I^{N}_i$. For these $n$, we use boundedness of $R$ and the inequality  $1-n\delta\ge 1-r_m-\delta>(1-r_m)/2$ to obtain 
\begin{equation*}
E_{1,n}\le (R(0)+K\delta)\left(1-\sqrt{\frac{1-n\delta-\delta}{1-n\delta}}\right)
%<R(0)\left(1-\sqrt{\frac{1-r_m-2\delta}{1-r_m-\delta}}\right)
< \frac{2(R(0)+K/4)\delta}{(1-r_m)}=:C_0\delta.
\end{equation*}
In summary, 
\begin{equation}\label{eq:fconvE1}
E_{1,n}\le
\begin{cases}
K\delta, & \text{if } n_0+1\le n\le N-2,\\
K\delta^{3/2}, & \text{if } n\le n_0 \text{ and $\left[n\delta,(n+1)\delta\right]\subseteq I^{N}_i $ for some $i$},\\
 C_0\delta , & \text{otherwise (at most $3m$ instances).}
\end{cases}
\end{equation}

\vspace{.5em}

\emph{Estimation of $E_{2,n}$.} Taylor's theorem implies that for all $x\in[0,1]$,
$$
 0\leq \sqrt{x} - \frac{2x}{1+x} \leq \frac{1}{2}(1-x)^{2}\left( \frac{1}{4x^{3/2}} + \frac{4}{(1+x)^3} \right).
$$
Using this with $x=1-\frac{\delta}{1-r}$ yields
\begin{align*}
\left|g(r)-f(r)\right|\le \frac{\delta^2}{(1-r)^2} h\left(1-\frac{\delta}{1-r}\right),
\end{align*}
where $h(x):=\frac{1}{8x^{3/2}} + \frac{2}{(1+x)^3}$.
For $0\le n\le Nr_m$, we have %$1-n\delta\ge 1-r_m$ and thus 
$\frac{\delta}{1-n\delta}\le \frac{\delta}{1-r_m}\le \frac{1}{2}$ and thus
\begin{align*}
E_{2,n}&\le  (R(0)+K/4)\frac{\delta^2}{(1-r_m)^2}h\left(\frac{1}{2}\right)=: C_1 \delta^2.
\end{align*}
For $Nr_m<n\le N-2$, we have $(n+1)\delta>r_m+\delta$ and it follows that
\begin{align*}
&\left|v\left(\frac{n+1}{N}\right)-{R}^{N}_{n+1}\right|\\
&\; =\left| \frac{1}{2\sqrt{1-(n+1)\delta}}\int_{(n+1)\delta}^1\!\!\!\frac{R(y)-R((n+1)\delta)}{\sqrt{1-y}}\,dy+R((n+1)\delta)-R^{N}_{n+1}\right|\\
%&\le  \frac{1}{2\sqrt{1-(n+1)\delta}}\int_{(n+1)\delta}^1\frac{\left|R(y)-R((n+1)\delta)\right|}{\sqrt{1-y}}dy+\left|R((n+1)\delta)-R^{N}_{n+1}\right|\\
&\;\le K(1-(n+1)\delta)+K\delta=K(1-n\delta).
\end{align*}
Consequently, 
\begin{align*}
E_{2,n}&\le K(1-n\delta)\frac{\delta^2}{(1-n\delta)^2} h\left(1-\frac{\delta}{1-n\delta}\right)\le \frac{K}{2}h\left(\frac{1}{2}\right) \delta<\frac{K}{2}\delta,
%&= K \left(\frac{1}{4(1-n\delta)}\left(\frac{1-n\delta}{1-n\delta-\delta}\right)^{3/2}+\frac{4(1-n\delta)^2}{\left(2(1-n\delta)-\delta\right)^3}\right)\delta^2\\
%&\le K \left(\frac{1}{4(1-n\delta)}\left(\frac{1-n\delta}{1-n\delta-\delta}\right)^{3/2}+\frac{1}{2(1-n\delta)-\delta}\frac{1-n\delta}{1-n\delta-\delta}\right)\delta^2\\
%&\le K \left(\frac{2^{3/2}}{8\delta}+\frac{2}{3\delta}\right)\delta^2<2K\delta,
\end{align*}
where we have used that $1-n\delta\ge 2\delta$.

As in the estimation of $E_{1,n}$, the bound can be improved if $Nr_m<n\le n_0$ and thus $1-n\delta\ge \sqrt{\delta}$. In this case, we have
\begin{align*}
E_{2,n}
%&\le K \left(\frac{1}{4(1-n\delta)}\left(\frac{1-n\delta}{1-n\delta-\delta}\right)^{3/2}+\frac{1}{2(1-n\delta)-\delta}\frac{1-n\delta}{1-n\delta-\delta}\right)\delta^2\\
%&\le K \left(\frac{1}{4\sqrt{\delta}}\left(\frac{1}{1-\sqrt{\delta}}\right)^{3/2}+\frac{1}{2\sqrt{\delta}-\delta}\frac{1}{1-\sqrt{\delta}}\right)\delta^2< 3K\delta^{3/2}.
&\le Kh\left(\frac{1}{2}\right) \delta^{3/2}< K \delta^{3/2}.
\end{align*}
In summary,
\begin{equation}\label{eq:fconvE2}
E_{2,n}\le
\begin{cases}
\frac{1}{2}K\delta, & \text{if } n_0+1\le n\le N-2,\\
K\delta^{3/2}, & \text{if } Nr_m<n\le n_0,\\
C_1 \delta^2 , & \text{if } n\le Nr_m.
\end{cases}
\end{equation}

\emph{Combining the Estimates.}
Combining \eqref{eq:fconvDelta}, \eqref{eq:fconvE1} and \eqref{eq:fconvE2} and recalling $\Delta_{N-1}\le K\delta$, we obtain for $n_0+1\le n\le N-2$ that
\[\Delta_n\le \Delta_{N-1}+\frac{3}{2}K\delta(N-1-n)\le K\delta+\frac{3}{2}K\sqrt{\delta}<\frac{5}{2}K\sqrt{\delta}\]
and for $n\le n_0$ that
\begin{align*}
\Delta_n&\le \Delta_{n_0+1}+(2K+C_1)\delta^{3/2}(n_0+1-n)+3mC_0 \delta\\
&\le \frac{5}{2} K\sqrt{\delta}+(2K+C_1)\sqrt{\delta}(1-\sqrt{\delta}+\delta)+3m C_0\delta<C_2\sqrt{\delta}.
\end{align*}
%\[
%\Delta_n\le \Delta_{n+1} +
%\begin{cases}
%3C\delta, & \text{if } n_0+1\le n\le N-2;\\
%3C\delta^{3/2}, & \text{if } Nr_m+1\le n\le n_0; \\
%KN^{-3/2}+K'N^{-2}, & \text{if } n\le Nr_m \text{ and $\left[n\delta,(n+1)\delta\right]\subseteq I^{N}_i $ for some $i$}\\
%R(0) N^{-1/2}+K' N^{-2}, &\text{otherwise}.
%\end{cases}
%\]
Putting everything together, we have
$\sup_{0\le n\le N}\Delta_n\le \frac{\max(5K/2,C_2)}{\sqrt{N}}.$
It remains to note that
\[|v'(r)|=\frac{|R(r)-v(r)|}{2(1-r)}\le 
\begin{cases}
\frac{K}{2}, & r_m<r< 1,\\[.2em]
\frac{R(0)}{2(1-r_m)}, & 0\le r\le r_m,
\end{cases}\]
and consequently,
\begin{align*}
\left|v^{N}_{\floor{rN}}-v(r)\right|
\le \Delta_{\floor{rN}}+ \left|v\left(\frac{\floor{rN}}{N}\right)-v\left(r\right)\right|
\le \Delta_{\floor{rN}}+\frac{\|v'\|_\infty}{N}<\frac{C_{3}}{\sqrt{N}}.
\end{align*}

Since $\lambda^{*}(r)=\frac{R(r)-v(r)}{2c(r)}$ and $\lambda^{N}_n=\frac{R^{N}_{n+1}-v^{N}_n}{2c^{N}_n}$, the convergence of $\lambda^{N}$ follows from the uniform convergence of the value functions and the cost coefficients, the almost uniform convergence of the reward scheme, and the uniform boundedness of $1/c$.
\end{proof}

\begin{proof}[Proof of Theorem~\ref{thm:principal-convergence}.]
  We first observe the convergence of the Riemann sum
  \[C^{N}:=\frac{1}{2N}\sum_{n=0}^{\ceil{\alpha N}-1} \sqrt{\frac{c^{N}_n(2-\frac{n+1}{N})}{(1-\frac{n}{N})(1-\frac{n+1}{N})}} \to \frac{1}{2} \int_0^\alpha \frac{ \sqrt{ c(r)(2-r)}}{(1-r)}\,dr =: C.\]
  The rate of convergence is $O(1/N)$ since $c^{N}$ converges to $c$ (which is bounded away from zero) uniformly at rate $O(1/N)$ and $\frac{\sqrt{2-r}}{1-r}$ is Lipschitz continuous on $[0,\alpha]$. Using the formulas of Theorems~\ref{th:principal} and~\ref{thm:optRN}, we conclude that
  \[\lim_{N\rightarrow\infty} ET^{N}_{\ceil{\alpha N}}=\lim_{N\rightarrow\infty}\frac{4(C^{N})^2}{B}=\frac{4C^2}{B}=T^*_\alpha\]
  with rate $O(1/N)$.
  Turning to the convergence of the reward schemes, we similarly observe that for $r\le \alpha$, 
  \[y_{\ceil{rN}-1}=\sqrt{\frac{c^{N}_{\ceil{rN}-1} N(N-\ceil{rN})}{(N-\ceil{rN}+1)(2N-\ceil{rN})}}\rightarrow \sqrt{\frac{c(r)}{2-r}}=y(r)\]
  and
  \[\frac{1}{N}\sum_{k={\ceil{rN}}-1}^{{\ceil{\alpha N}}-1}\frac{ y_k}{1-\frac{k+1}{N}}\rightarrow \int_r^\alpha \frac{y(s)}{1-s}\,ds,\]
  uniformly in $r\in [0,\alpha]$ with rate $O(1/N)$. Hence, the formulas in Theorems~\ref{th:principal} and~\ref{thm:optRN} yield that
  \begin{align*}
   R^{N}_{\ceil{r N}} &=\frac{B}{C^{N}}\left\{y_{\ceil{r N}-1}+\frac{1}{2}\sum_{k=\ceil{r N}-1}^{\ceil{\alpha N}-1}\frac{ y_k}{N-k-1}\right\}\\
  &\rightarrow \frac{B}{C}\left\{y(r)+\frac{1}{2}\int_r^\alpha \frac{y(s)}{1-s}\,ds\right\}=R^*(r)
  \end{align*}
  uniformly in $r\in [0,\alpha]$ with rate $O(1/N)$.
\end{proof}

\begin{proof}[Proof of Corollary~\ref{co:epsOptimalPrincipal}.]
  Theorem~\ref{thm:fwd-convergence} and Remark~\ref{rmk:better-rate} imply that the $N$-player equilibrium control $\lambda^{(N)}_{\floor{rN}}$ for $R^{(N)}$ converges uniformly to the mean field equilibrium control $\lambda^*(r)$ for $R^*$, at rate $O(1/N)$. It follows that 
  \[f_N(r):=(1-\floor{rN})\lambda^{(N)}_{\floor{rN}}\to (1-r)\lambda^*(r)=:f(r)\]
  uniformly at rate $O(1/N)$. Since $\lambda^*$ is uniformly bounded away from zero on $[0,\alpha]$, the same holds for $f$ and $f_N$  with large $N$. Thus, using~\eqref{eq:ET} and~\eqref{eq:Talpha},
  \begin{align*}
  \left|ET^{(N)}_{\ceil{\alpha N}} -T^*_\alpha\right|&= \left|\sum_{n=0}^{\ceil{\alpha N}-1}\frac{1}{(N-n)\lambda^{(N)}_n}-\int_0^\alpha \frac{dr}{(1-r)\lambda^*(r)}\right|\\
  %&=\left|\int_0^\frac{\ceil{\alpha N}}{N} \frac{dr}{(1-\floor{rN})\lambda^{N}_{\floor{rN}}}-\frac{dr}{(1-r)\lambda^*(r)}\right|\\
  &\le \int_0^\frac{\ceil{\alpha N}}{N}  \frac{|f(r)-f_N(r)|}{f_N(r) f(r)} dr=O(1/N).
  \end{align*}
  As Theorem~\ref{thm:principal-convergence} shows that $|T^*_\alpha -ET^{N}_{\ceil{\alpha N}}|=O(1/N)$, the claim follows.
%  We may now use Theorem~\ref{thm:principal-convergence} to conclude that
%  \[
%    \left|ET^{(N)}_{\ceil{\alpha N}} - ET^{N}_{\ceil{\alpha N}}\right|\le \left|ET^{(N)}_{\ceil{\alpha N}}-T^*_\alpha\right|+\left|T^*_\alpha -ET^{N}_{\ceil{\alpha N}}\right|=O(1/N).
%  \]
\end{proof}

\section{Exact Law of Large Numbers}\label{se:exactLLN}\label{se:ELLN}

In this section, we detail a setting such that the exact law of large numbers holds for a continuum of essentially pairwise independent random variables. 
Let $(I,\cI,\mu)$ be an atomless (hence uncountable) probability space and let $(\Omega,\cF,P)$ be another probability space.

\begin{definition}\label{de:essentialPairwiseIndep}
  A family $(f_{i})_{i\in I}$ of random variables on $(\Omega,\cF,P)$ is \emph{essentially pairwise  independent} if for $\mu$-almost all $i\in I$, $f_{i}$ is independent of $f_{j}$ for $\mu$-almost all $j\in I$. The family is \emph{essentially pairwise i.i.d.}\ if, in addition, all $f_{i}$ have the same distribution.
\end{definition}

In what follows, we need to work on a probability space that is larger than the usual product\footnote{
Here we use the convention that the product $\sigma$-field $\cI\otimes \cF$ is completed. 
}
 $(I\times\Omega,\cI\otimes \cF,\mu\otimes P)$, because the latter does not support relevant families of i.i.d.\ random variables (see, e.g., \cite[Proposition~2.1]{Sun.06}).
%
%\begin{remark}\label{rk:productDegenerate}
%  If $f: I\times\Omega\to\R$ is an $\cI\otimes \cF$-measurable function such that $f(i,\cdot)$, $i\in I$ are essentially pairwise i.i.d., then $f$ is constant $\mu\otimes P$-a.s.
%\end{remark}
Following Sun~\cite{Sun.06}, a probability space $(I\times\Omega,\Sigma,\nu)$ is an \emph{extension} of the  product $(I\times\Omega,\cI\otimes \cF,\mu\otimes P)$ if $\Sigma$ contains $\cI\otimes \cF$ and the restriction of $\nu$ to $\cI\otimes \cF$ coincides with $\mu\otimes P$. It is a \emph{Fubini extension} if, in addition, any $\nu$-integrable\footnote{
That is, $f$ is measurable for the $\nu$-completion of $\Sigma$ and $\int |f|\,d\nu<\infty$.
}
function $f: I\times\Omega\to\R$ satisfies the assertion of Fubini's theorem;
%\footnote{
%Since $\Sigma$ may be strictly larger than $\cI\otimes \cF$, this is not automatic.
%}
that is,

\begin{enumerate}
 \item for $\mu$-almost all $i\in I$, the function $f(i,\cdot)$ is $P$-integrable,
 \item for $P$-almost all $\omega\in \Omega$, the function $f(\cdot,\omega)$ $\mu$-integrable,
 \item $i\mapsto \int f(i,\cdot)\,dP$ is $\mu$-integrable, $\omega\mapsto \int (\cdot,\omega)\,d\mu$ is $P$-integrable, and
 $$
   \int f\,d \nu = \iint f(i,\omega)\,P(d\omega)\,\mu(di) = \iint f(i,\omega)\,\mu(di)\,P(d\omega).
 $$
\end{enumerate}

Let $(I\times\Omega,\Sigma,\nu)$ be a Fubini extension of $(I\times\Omega,\cI\otimes \cF,\mu\otimes P)$. Then, essentially pairwise independent families satisfy an exact version of the Law of Large Numbers. The following is a special case of \cite[Corollary~2.9]{Sun.06}.

%\begin{proposition}[Exact Law of Large Numbers]\label{pr:LLN}
%  Let $S$ be a Polish space and let $f: I\times\Omega\to S$ be $\Sigma$-measurable. If $f(i,\cdot)$, $i\in I$ are essentially pairwise independent, then for $P$-almost all $\omega\in\Omega$, the sample distribution $\mu\circ f(\cdot,\omega)^{-1}$ coincides with the distribution $\nu\circ f^{-1}$.
%  
%  In particular, if $S=\R$ and $f(i,\cdot)$, $i\in I$ are essentially pairwise i.i.d.\ with a distribution having finite mean $m$, then $\int f(\cdot,\omega)\,d\mu=m$ for $P$-almost all $\omega\in\Omega$.
%\end{proposition}

\begin{proposition}[Exact Law of Large Numbers]\label{pr:ELLN}
  Let $f: I\times\Omega\to \R$ be $\nu$-integrable. If $f(i,\cdot)$, $i\in I$ are essentially pairwise i.i.d.\ with a distribution having mean $m$, then $\int f(\cdot,\omega)\,d\mu=m$ for $P$-almost all $\omega\in\Omega$.
\end{proposition}  

%We shall also need a conditional version as provided by \cite[Corollary~2]{QiaoSunZhang.14}.
%
%\begin{proposition}[Conditional Exact Law of Large Numbers]\label{pr:condLLN}
%  Let $\cC\subseteq\cF$ be a countably generated $\sigma$-field and let $f: I\times\Omega\to \R$ be $\nu$-integrable. If $f(i,\cdot)$, $i\in I$ are essentially pairwise conditionally independent given $\cC$, then $\int f(\cdot,\omega)\,d\mu=\int E^{\nu}[f|\cI\otimes\cC](\cdot,\omega)\,d\mu$ for $P$-almost all $\omega\in\Omega$. 
%\end{proposition}
%
%In view of Remark~\ref{rk:productDegenerate}, it is not obvious that the preceding propositions are not vacuous---that is guaranteed by the next two results.

Next, we turn to the existence of such a setting. The space $(I\times\Omega,\Sigma,\nu)$ is called \emph{rich} if there exists a $\Sigma$-measurable function $f: I\times\Omega\to\R$ such that $f(i,\cdot)$, $i\in I$ are essentially pairwise i.i.d.\ with a uniform distribution on $[0,1]$. By composing $f$ with a suitable function, it then follows that $(I\times\Omega,\Sigma,\nu)$ supports essentially pairwise i.i.d.\ families with any given distribution.%; cf.\ \cite[Corollary~5.4]{Sun.06}.
%
%\begin{lemma}\label{le:richSupportsiid}
%  Let $(I\times\Omega,\Sigma,\nu)$ be a rich Fubini extension of $(I\times\Omega,\cI\otimes \cF,\mu\otimes P)$, let $S$ be a Polish space and let $\nu$ be a Borel probability measure on $S$. There exists a $\Sigma$-measurable function $f:I\times\Omega\to S$ such that $f(i,\cdot)$, $i\in I$ are essentially pairwise independent and $f(i,\cdot)$ has distribution $\nu$ for all $i\in I$.
%\end{lemma}
%We still need to state that a rich Fubini extension space actually exists.

\begin{lemma}\label{le:richExtensionsExist}
  There exist atomless probability spaces $(I,\cI,\mu)$ and $(\Omega,\cF,P)$ such that $(I\times\Omega,\cI\otimes \cF,\mu\otimes P)$ admits a rich Fubini extension.
\end{lemma}

This is part of the assertion of \cite[Proposition~5.6]{Sun.06} which also shows that one can take $I=[0,1]$ and $\Omega=\R^{[0,1]}$. The main result of Sun and Zhang~\cite{SunZhang.09} shows that, in addition, one can take $\mu$ to be an extension of the Lebesgue measure (but not the Lebesgue measure itself). A different construction is presented by Podczeck~\cite{Podczeck.10}.

%%%%%%%%%%%%%%%%%%%%%%%%%%%%%%%%%%%%%%%%%%%%%%%%%%%%%%%%%%%%%%%%%
\bibliography{competition}
\bibliographystyle{plain}
%%%%%%%%%%%%%%%%%%%%%%%%%%%%%%%%%%%%%%%%%%%%%%%%%%%%%%%%%%%%%%%%%

\end{document}